\documentclass[12pt,a4paper]{amsart}
\usepackage{amsmath}
\usepackage{amsthm}
\usepackage{amssymb}
\usepackage{mathrsfs}
\usepackage[british]{babel}
\usepackage{ot1patch}
\usepackage{graphicx}
\usepackage{verbatim}
\usepackage{subcaption}
\usepackage[leftcaption]{sidecap}
\input{insbox}%%%%%%%%%%%%%% TeX macro,

\newtheorem{thm}{Theorem}[section]
\newtheorem{lem}[thm]{Lemma}
\newtheorem{cor}[thm]{Corollary}
\newtheorem{prop}[thm]{Proposition}

\theoremstyle{remark}
\newtheorem{rem}[thm]{Remark}
\newtheorem*{rem*}{Remark}

\theoremstyle{definition}
\newtheorem{dfn}[thm]{Definition}
\newtheorem{ex}[thm]{Example}

\numberwithin{equation}{section}

\newcommand{\Rz}{\mathbb{R}}

\graphicspath{ {Graphics/} }

\title{The tangent cone, the dimension and the frontier of a medial axis}
\author{Adam Białożyt}
\begin{document}

\maketitle
\begin{abstract}
This paper aims to to establish a relation between the tangent cone of the medial axis of $X$ at a given point $a\in\Rz^n$ and the medial axis of the set of points in $X$ realising the distance $d(a,X)$. As a consequence, a lower bound for the dimension of the medial axis of $X$ in terms of the dimension of the medial axis of $m(a)$ is obtained. This appears to be the missing link to the full description of the medial axis' dimension. Further study of potentially troublesome points on the frontier of the medial axis is also provided, resulting in their characterisation in terms of the reaching radius.
\end{abstract}
\section{Introduction}

The medial axis, introduced by Blum in \cite{Blum} as a central object in pattern recognition, under various names emerges in a number of mathematical and application problems. Lossless compression of data makes it an appealing object in tomography, robotics, or simulation. On the other hand, its natural definition appears in various versions also in the fields of partial differential equations or convex analysis. A deep connection with the initial set's geometry makes it an interesting object for the study of geometrical and topological properties such as its singularities or homotopy groups.

As medial axes are closely related to the extensively developed notion of conflict sets, one can believe that most of the theorems concerning conflict sets should have their counterparts in the medial axis theory. Unfortunately, the medial axes are (in)famous for their instability, thus the proofs are seldom transferable between the theories concerning these two objects. The main result of the present paper is a proof of the medial axis analog of \cite{BirbrairSiersma} Theorem 2.2. Since the proof presented by Birbrair and Siersma depends heavily on the monotonicity of Conflict Sets -- a phenomenon that has no counterpart in the medial axis setting -- we are forced to develop a completely new approach to the problem based on an analysis of the graph of the distance function. A problem similar to the one presented in the paper was studied in a slightly broader sense and on the grounds of the convex analysis in \cite{Cannarsa}. Focusing precisely on the medial axis, we are able to provide more rigid results and formulas. An immediate application of the result gives an answer to the question of the dimension of a medial axis raised in \cite{Erdos} and \cite{DenkowskiOnPoints}. Later in the paper, potentially troublesome points of the medial axis' frontier are characterised by a limiting directional reaching radius, an object combining the virtues of Birbrair-Denkowski reaching radius \cite{BirbrairDenkowski} and Miura's radius of curvature \cite{Miura}. It is adapted for the study of sets with higher codimension and it also provides insights on the Birbrair-Denkowski archetype.

The attention of this paper is restricted to sets that are definable in o-minimal structures expanding the field of real numbers. Such an approach gives us a framework with an appealing definition of dimension (coinciding with the Hausdorff dimension) and a handful of useful tools. At the same time, it protects us from pathological sets while conserving the applicability of the setting. Readers who are not familiar with the notion of definable sets may think of them as semialgebraic sets. An excellent introduction to the notion can be found in \cite{Coste} or \cite{Rolin}. 

Whenever in the paper the continuity (or upper- and lower limits) of a (multi-)function is mentioned, it refers to the continuity (or upper- and lower limits) in the Kuratowski sense (more on the Kuratowski convergence can be found in the book \cite{RockafellarWets}, whereas an introduction to its relation with medial axes is given in \cite{DenkowskiLimits}). For $x,y\in\mathbb{R}^n$ we denote by $[x,y]$ the closed segment joining $x$ and $y$. The closed ball centered at $a$, of radius $r$, is denoted by $\mathbb{B}(a,r)$, and $\mathbb{S}(a,r)$ denotes its boundary -- an $(n-1)$-dimensional sphere of radius $r$ centered at $a$.

For a closed, nonempty subset $X$ of $\mathbb{R}^n$ endowed with the Euclidean norm, we define the distance of a point $a\in\mathbb{R}^n$ to $X$ by 
$$d(a,X)=d_X(a):= \inf\lbrace \|a-x\| |\,x\in X\rbrace,$$

which allows us to introduce the set of closest points in $X$ to $a$ as
$$m_X(a):=\lbrace x\in X|\, d(a,X)=\|a-x\|\rbrace.$$
We will usually drop the indices of the (multi-)functions $d$ and $m$.

The main object discussed in this paper is the \textit{medial axis} of $X$ denoted by $M_X$, that is, the set of points of $\mathbb{R}^n$ admitting more than one closest point in the set $X$, namely
$$M_X:=\lbrace a\in\mathbb{R}^n|\, \# m(a)>1\rbrace.$$
A descriptive way, the most often invoked, to imagine the medial axis, brings an image of the propagation of a fire front starting at $X$. The medial axis of $X$ in this case is exactly the set of points where fronts originating from different starting points meet. This picturesque idea illustrates maybe the most profound feature of the medial axis -- it collects exactly those points of the ambient space, at which the distance function is not differentiable.

\begin{figure}[h]
    \centering
    \begin{subfigure}{0.32\textwidth}
        \includegraphics[width=\textwidth]{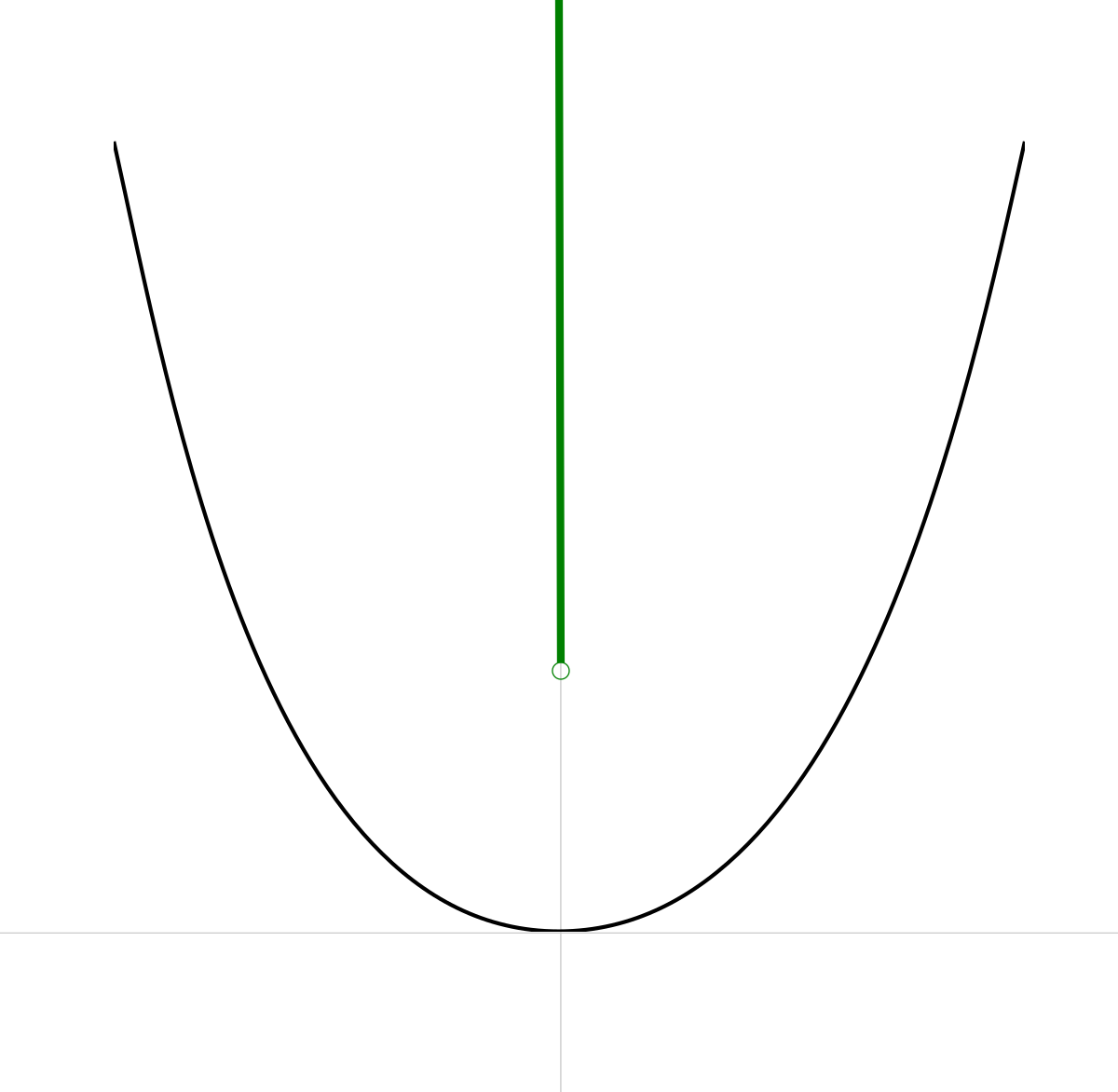}
        \caption{}
        \label{fig:sfig1}
    \end{subfigure}    
    \begin{subfigure}{0.32\textwidth}
        \includegraphics[width=\textwidth]{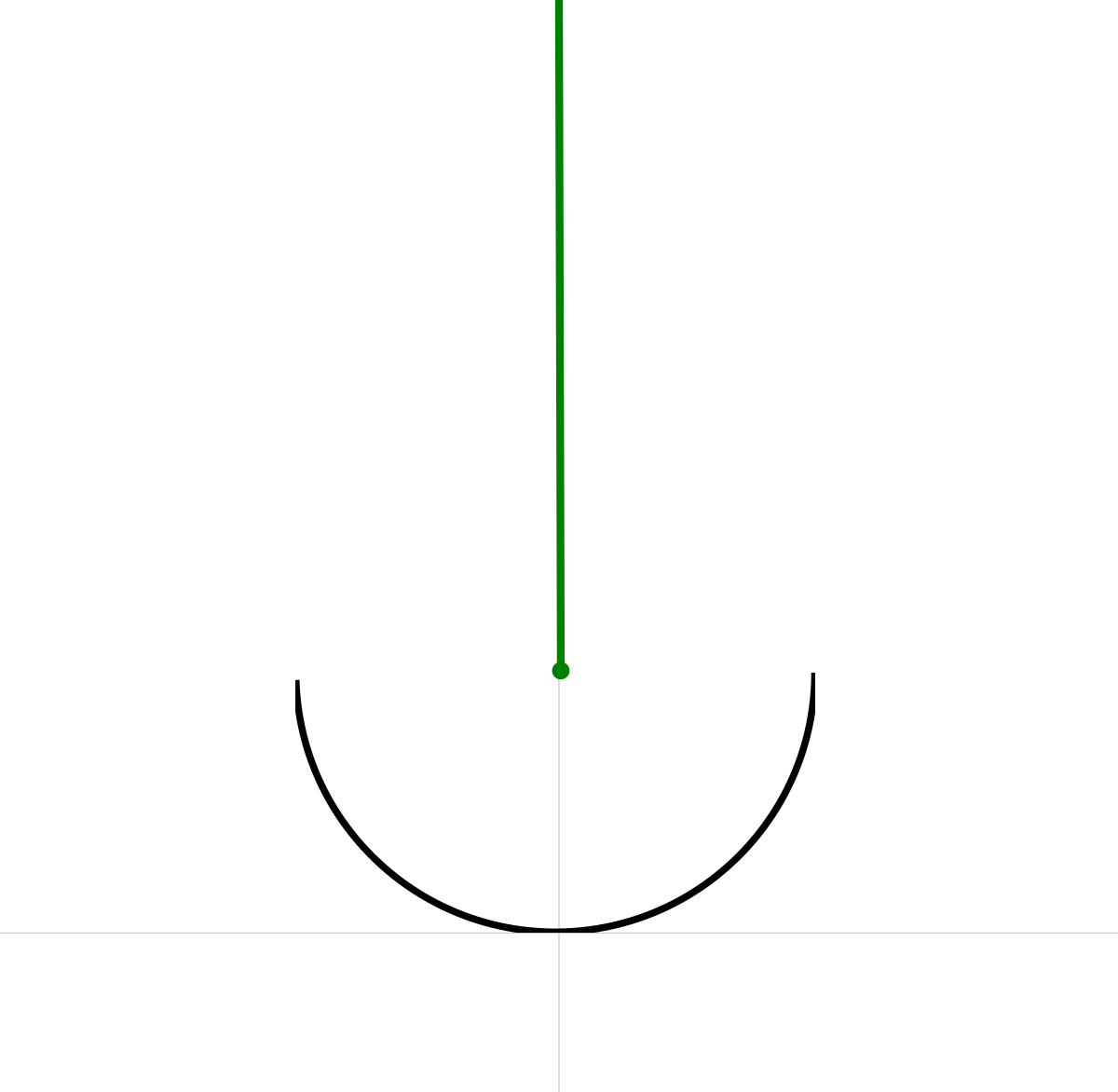}
        \caption{}
        \label{fig:sfig2}
    \end{subfigure}  
    \begin{subfigure}{0.32\textwidth}
        \includegraphics[width=\textwidth]{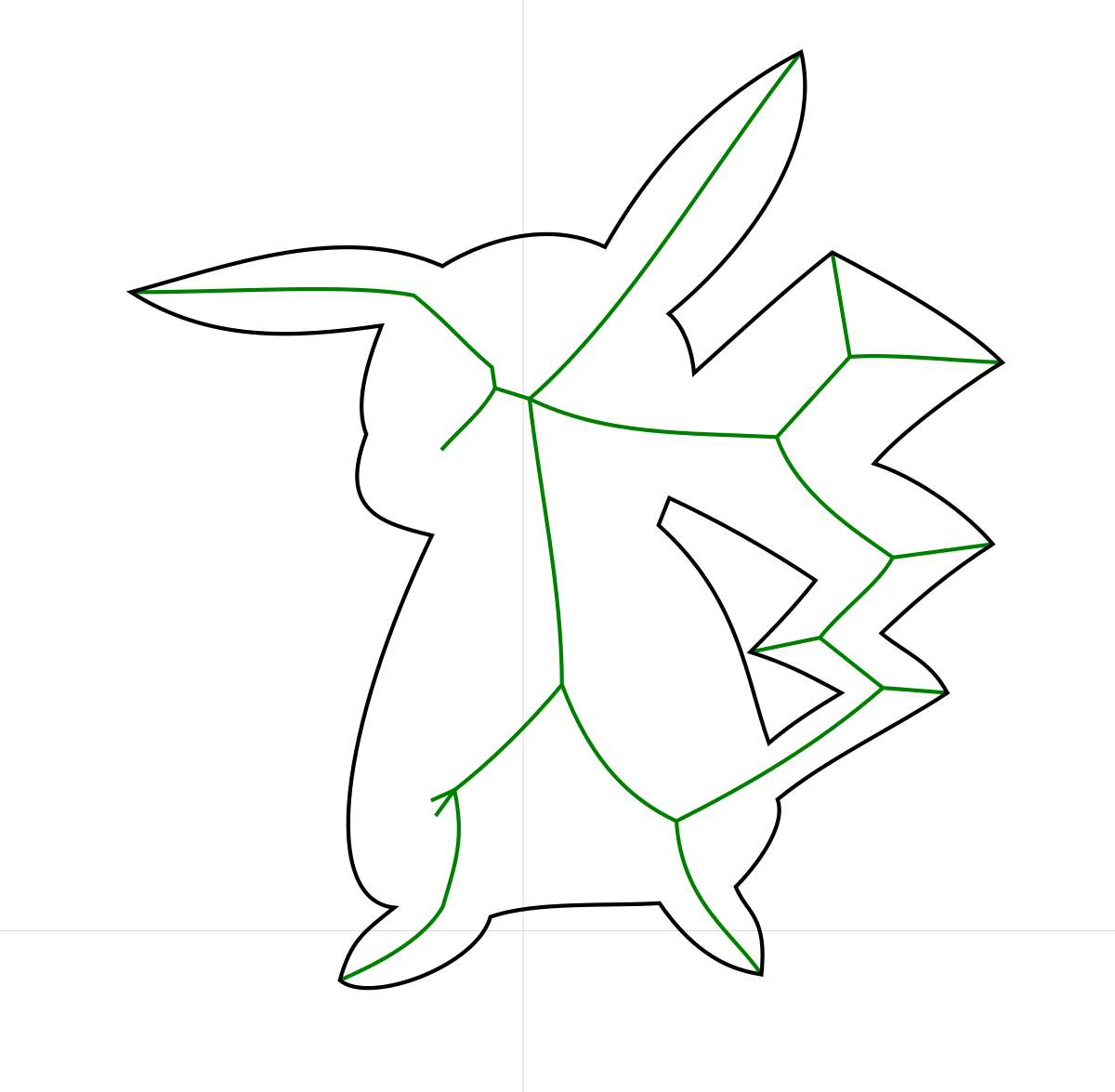}
        \caption{}
        \label{fig:sfig3}
    \end{subfigure}
    \caption{Examples of medial axes (in green) of euclidean plane subsets. (A) A graph of the function $y=x^2$ (B) A graph of the function $y=\sqrt{1-x^2}$. (C) A silhouette of a Pikachu. }
    \label{fig:my_label}
\end{figure}

%jeżeli chciałbym wpleść równanie eikonalu to tutaj zdaje się jest na to dobre miejsce. 

As an introductory remark, it is worth recalling that as was shown in \cite{DenkowskiOnPoints}, both the medial axis and the multifunction $m(x)$ are definable in the same structure as $X$. Moreover, the multifunction $m(x)$ is \textit{upper semi-continuous}, meaning $\limsup_{A\ni a\to a_0} m(a)\subset m(a_0)$ for any set $A$ with $a_0$ in its closure.

\section{The tangent cone of the medial axis}
Let us begin by recalling that we have an explicit formula for the directional derivative of the distance function coming from Richard von Mises \cite{Mises} (most often misquoted as M.R. de Mises, see also \cite{Zajicek}).

\begin{thm}[R. von Mises]\label{Mises}
Let $X$ be a closed, nonempty subset of $\mathbb{R}^n$, then for every point  $a\in \mathbb{R}^n\backslash X$ all one-sided directional derivatives of the distance function $d_X$ exist and are equal to 
$$D_v d_X(a)=\inf \lbrace -<v,\frac{x-a}{\|x-a\|}>,\, x\in m(a)\rbrace. $$
\end{thm}

\begin{proof}

For any $a,b\in\mathbb{R}^n$ there is $<a,b>=\|a\|\|b\|\cos \alpha$ where $\alpha$ is an angle between $a$ and $b$. Thus, for $\|v\|=1$ the assertion can be written as 
$$D_v d_X(a)=\inf \lbrace -\cos\alpha_x,\, x\in m(a)\rbrace $$

where $\alpha_x$ is the angle formed between $x-a$ and $v$. Clearly, the value of $-\cos\alpha$ will be the smallest for the smallest $\alpha$. 
Without loss of generality, assume $v=(1,0,\ldots,0)$ and $a=0$, then take $x_0\in m(0)$ forming the smallest angle with $v$, and $x_t\in m(tv)$ for $t>0$. Since $\|x_t-tv\|\leq \|x_0-tv\|$ we obtain 
$$\|x_t\|^2-\|x_0\|^2\leq 2t(x_{t}^{(1)}-x_{0}^{(1)}),$$
where $x_{t}^{(1)}$ is the first coordinate of $x_t$. Note that since $x_t\notin int \,\mathbb{B}(0,d(0))$ there is in particular 
$$0\leq x_{t}^{(1)}-x_{0}^{(1)}.$$
Denote now by $\alpha_t$ an angle formed by $v$ and $x_t$. By the cosinus theorem for a triangle formed by $tv,0,x_t$, we have 
$$d(tv)^2=\|x_t\|^2+t^2-2\|x_t\|t\cos \alpha_t.$$
Keeping in mind $d(0)=\|x_0\|$ we can clearly see that 
$$ \frac{d(tv)-d(0)}{t}=\frac{1}{d(tv)+d(0)}\left(\frac{\|x_t\|^2-\|x_0\|^2}{t}+t-2\|x_t\|\cos\alpha_t\right).$$
Both $d(tv)$ and $\|x_t\|$ converge to $d(0)$ as $t\rightarrow 0$, so the proof will be completed if only $\alpha_t\rightarrow\alpha_0$ and  $\frac{\|x_t\|^2-\|x_0\|^2}{t}\rightarrow 0$. However, both claims can be derived from the closedness of $X$. Indeed, closedness guarantees that for any $\varepsilon >0$, we can find such $\delta>0$ that for all $x\in X$ with the first coordinate greater than or equal to $x_{0}^{(1)}$  there is $x^{(1)}-x_{0}^{(1)}\leq \varepsilon$ as long as $\|x\|<d(0)+\delta$, otherwise $x_0$ would not realise the smallest angle among those formed by $v$ and $x\in m(0)$.

\end{proof}

\begin{figure} 
    \begin{subfigure}{0.4\textwidth}
        \includegraphics[width=\textwidth]{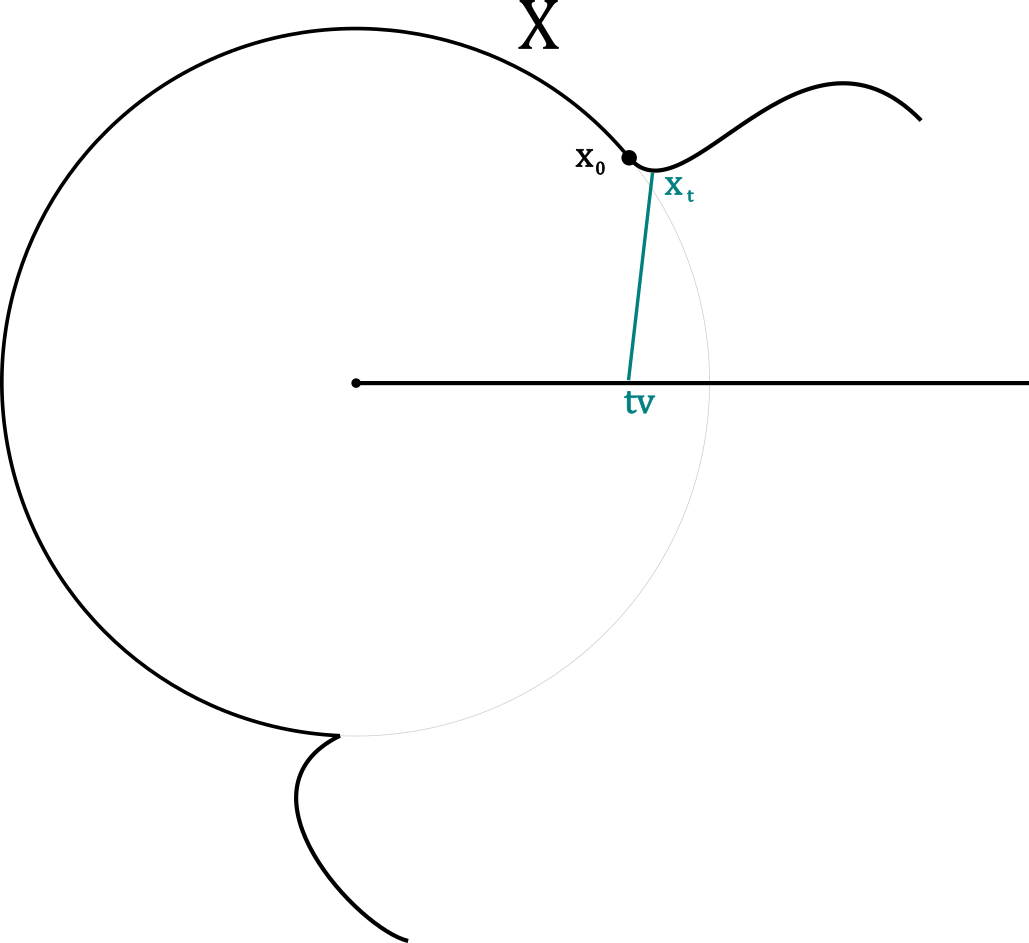}
        \caption{}
    \end{subfigure}    
    \begin{subfigure}{0.4\textwidth}
        \includegraphics[width=\textwidth]{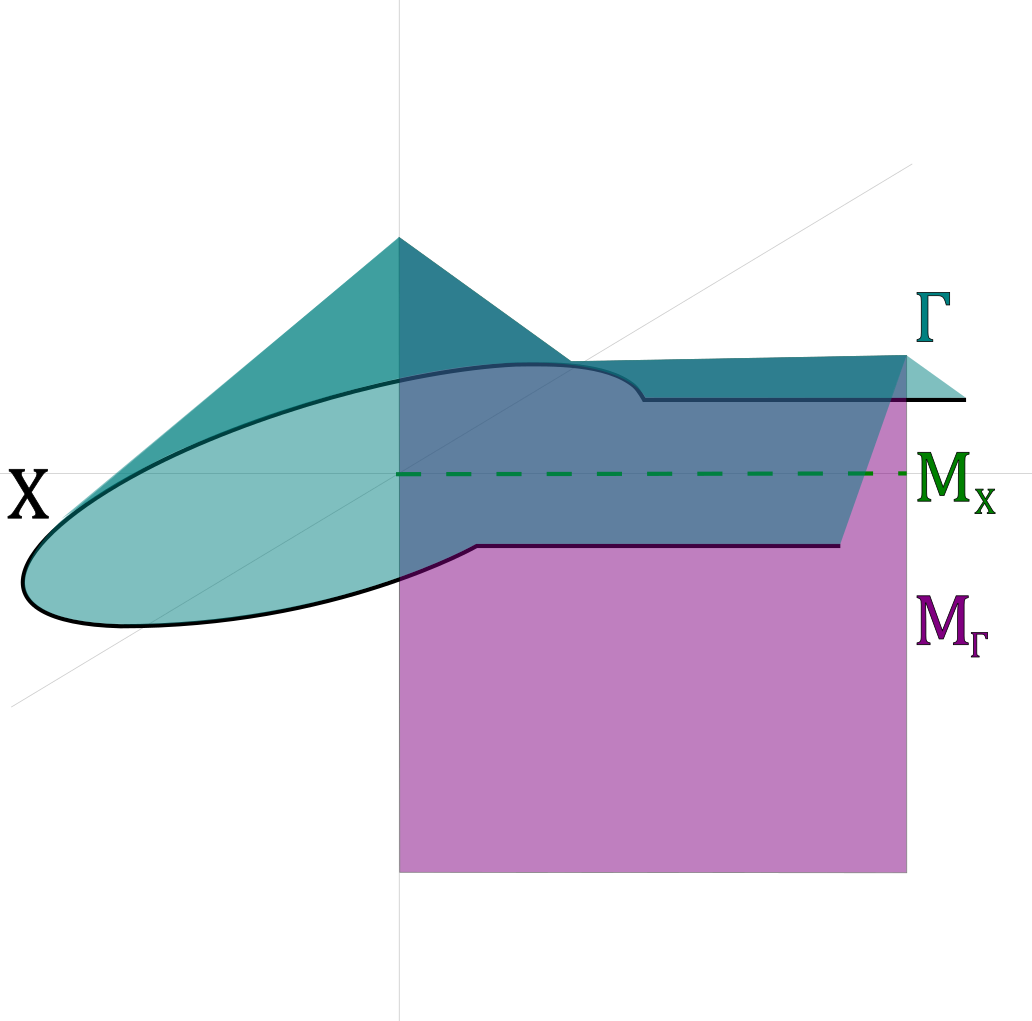}
        \caption{ }
    \end{subfigure}    
    \caption{(A) A graphic depiction of Theorem ~\ref{Mises}. (B) A graph $\Gamma$ of the distance function for a set $X=\partial (\mathbb{B}(0,2)\cup \lbrace |y|<1,\, x>0\rbrace)$ together with its medial axis. }
\end{figure}

Assuming that $m(0)$ is a subset of the unit sphere, we can use the polarization identity for the inner product to express $D_v d(0,X)$ in a convenient form $$ \frac{1}{2}\inf \lbrace \|v-y\|^2-\|v\|^2-1,\, y\in m(0)\rbrace.$$
Since the infimum is attained at $y\in m(0)$ which is closest to $v$, the formula simplifies even further down to  
$$\frac{1}{2}(d(v,m(0))^2-\|v\|^2-1).$$

The appearance of $d_{m(0)}$ in the formula for $D_vd_X(0)$ brings interesting consequences and possibilities to describe the medial axis' cone within its own category. However, our first result will be independent of the Mises Theorem, for the situation of the subsets of the sphere is simpler than the general one.

\begin{prop}\label{Szkielet podzbioru sfery}
Let $Y\subset \Rz^n$ be a closed proper subset of the unit sphere $\mathbb{S}$, then $M_Y$ is a cone spanned over $M^{\mathbb{S}}_{Y}$ calculated in $\mathbb{S}$ with the induced norm. Moreover, in that case $C_0M_Y=\overline{M_Y}$.
\end{prop}
\begin{proof}

Start by observing that since $Y$ is a subset of $\mathbb{S}$ and an open ball in the induced norm is an intersection of a ball in $\mathbb{R}^n$ with $\mathbb{S}$, clearly $M^{\mathbb{S}}_{Y}\subset M_Y$. What is left to prove is that $M_Y$ is a cone. This follows from the observation that the intersection of $\mathbb{S}$ and any ball $\mathbb{B}(x,d(x,m(x))$, is always equal to a closed ball in $\mathbb{S}$ centered at $x/\|x\|$ with radius $d(x/\|x\|,m(x))$. 

The second part of the theorem is trivial, as $M_Y$ is a cone, and every vector $\lambda v, \lambda\in\Rz_+, v\in M_Y$ approximating an element of a tangent cone $C_0M_Y$ is in $M_Y$.
\end{proof}

Even this simple situation demands some delicacy. For one, it is indeed necessary to take the closure of $M_Y$ in Proposition~\ref{Szkielet podzbioru sfery}. As seen for $Y=\lbrace xy=0\rbrace\cap\mathbb{S}\subset\mathbb{R}^3$, a point $v=(0,0,1)$ lies in $Y$ so it cannot be in $M_Y$. Nevertheless, it is easy to check that $v$ is a point of the tangent cone at the origin of the medial axis.
Additionally, since closed balls both in the spherical and in the induced norm are the same, the choice of the norm used to calculate $M^\mathbb{S}_Y$ in Proposition~\ref{Szkielet podzbioru sfery} does not affect the assertion.

Before we venture forth, we will prove a technical lemma helpful in the description of the geometry of a graph of the distance function.
\begin{lem}\label{wlasnosci}
For any closed $X\subset\mathbb{R}^n$ a graph $$\Gamma:=\lbrace (x,y)\in\mathbb{R}^n\times\mathbb{R}|\,y=d(x,X)\rbrace$$ has the following properties:
\begin{enumerate}
    \item For any $ (a,d(a))\in \Gamma$, 
    
    $\lbrace (x,y)\in\Rz^n\times \Rz|\; |y-d(a)|>\|x-a\|\rbrace\cap \Gamma =\emptyset$;
    \item For any $ a\in \Rz^n$ and $v\in m(a)$,   $[(v,0),(a,d(a))]\subset\Gamma$;
    \item For any $ (x,y)\in\Rz^n\times\Rz$ with $y<d(x)$, 
    
    $(x,y)\in M_\Gamma\iff x\in M_X$, in other words: the medial axis of the epigraph of d is equal to $M_X\times \Rz\cap \lbrace (x,y)|\,y<d(x)\rbrace.$
\end{enumerate}
\end{lem}
\begin{proof}
(1) is a  consequence of the Lipschitz condition for the distance function. 

(2) comes from $m(tv+(1-t)a)=\lbrace v\rbrace$ for $t\in(0,1],v\in m(a)$ together with $$\|tv+(1-t)a-v\|=\|(1-t)a-(1-t)v\|=(1-t)\|a-v\|.$$ 

(3) can be proved by observing that (1) and (2) give together: 
$$\lbrace (x,y)\in\Rz^n\times \Rz|\; y-d(a)\leq-\|x-a\|\rbrace\cap \Gamma =\bigcup_{v\in m(a)} [(v,0),(a,d(a))]$$
 for every $(a,d(a))\in \Gamma$. Indeed, if $(x,y)\in\bigcup [(v,0),(a,d(a))]$ then (2) gives $(x,y)\in \Gamma$, moreover for a certain $t\in [0,1]$ there is
 $$(x,y)=t(v,0)+(1-t)(a,d(a))=(tv+(1-t)a,(1-t)d(a)).$$
 Thus,
 $$y-d(a)+\|x-a\|=-td(a)+\|tv-ta\|=t(-d(a)+\|a-v\|)\leq 0,$$
as pleaded.

On the other hand, taking $(x,y)\in \lbrace y-d(a)\leq-\|x-a\|\rbrace\cap \Gamma$  by (1) there is 
$$d(x)-d(a)=y-d(a)=- \|x-a\|.$$ 
Therefore, for any $v\in m(a)$ and $x'\in m(x)$ there is
$$\|a-x'\|\leq\|x-x'\|+\|x-a\|=\|a-v\|$$
Thus, $m(x)\subset m(a)$ and since $y<d(a)$ there must exist $t\in [0,1]$ such that $(x,y)=(tv+(1-t)a,(1-t)d(a))$ for a certain $v\in m(a)$.

It is easy to check now that for every point $p$ of an axis of a cone $C(a):=\lbrace (x,y)|\, y-d(a)\leq -\|x-a\|\rbrace$, a set $m_\Gamma(p)$ has the same number of points as $m(a)$. Truly, whenever $m(a)$ is a singleton, an intersection $C(a)\cap\Gamma$ is a single segment on the boundary of $C(a)$, thus $m_\Gamma(p)=M_{C(a)\cap\Gamma}(p)$ must be an singleton as well. If $m(a)$ consists of more than one point on the other hand, then the intersection $C(a)\cap\Gamma$ is a sum of segments on the boundary of $C(a)$ with endpoints at the vertex of $C(a)$ and points of $m(a)\times \lbrace 0\rbrace$. Therefore, 
$$m_\Gamma(p)=\begin{cases}
m(a)\times\lbrace 0\rbrace, & p^{(n+1)}\leq-d(a),\\
\frac{d(a)-p^{(n+1)}}{2d(a)}m(a)\times \lbrace \frac{d(a)+p^{(n+1)}}{2}\rbrace, & p^{(n+1)}>-d(a)
\end{cases},$$
where $p^{(n+1)}<d(a)$ denotes  the last coordinate of $p$.

%Together with Proposition 4.1 from  \cite{BirbrairDenkowski} we conclude now that an axis of a cone $\lbrace y\leq -\|x\|\rbrace$ translated by $(a,d(a))$ is a subset of $M_\Gamma$ if and only if $a$ belongs to $M_X$.

\end{proof}
With the properties of the graph of the distance function at hand, we are ready to prove.
\begin{thm}\label{Stozek}
For any closed definable $X\subset \Rz^n$ assuming $0\in \overline{M_X}$ there is $M_{m(0)}\subset C_0M_X$.
\end{thm}
\begin{proof}
If $0\in \overline{M_X}\backslash M_X$, the theorem is trivial as $M_{m(0)}=\emptyset$. 
Assume then, without loss of generality, that $0\in M_X$ and $d(0)=1$. Denote by $\Gamma$ the graph of the distance function $d$ as was done in the previous lemma.
As was shown in \cite{BirbrairDenkowski} during the proof of Theorem 4.6,  $M_{C_{(0,1)}\Gamma}\subset C_{(0,1)}M_\Gamma$. To prove the theorem, we will establish the relation between these sets and $M_{m(0)}$ and $C_0M_X$.

Lets begin with $C_{(0,1)}M_\Gamma$. Since Lemma~\ref{wlasnosci}(3) gives  $$(M_\Gamma-(0,1))\cap \lbrace y\leq -\|x\|\rbrace = M_X\times \Rz \cap\lbrace y\leq-\|x\|\rbrace,$$ the tangent cones of $M_\Gamma-(0,1)$ and $M_X\times \Rz$ must coincide in a cone $\lbrace y\leq \alpha\|x\|\rbrace$ for any choice of $\alpha<-1$. Because $\Rz$ is a cone, we further obtain the coincidence of $C_{(0,1)}M_\Gamma$ and $C_0 M_X\times \Rz$ in the aforementioned cone.

As it comes to $M_{m(0)}$ and $M_{C_{(0,1)}\Gamma}$, we will study first a set $C_{(0,1)}\Gamma$. Since $d_X$ is a Lipschitz function, the explicit formula for a directional derivative $D_xd_X(0)$ allows us to express $C_{(0,1)}\Gamma$ as a graph of a function
$$x\rightarrow D_xd_X(0)= \frac{1}{2}(d(x,m(0))^2-\|x\|^2-1).$$ 
Consider for a moment a graph $\Gamma_1$ of a function $x\rightarrow d(x,m(0))$. For $\|x\|<1$ it has a structure of a cone with a vertex at $(0,1)$, furthermore a tangent cone $C_{(0,1)}\Gamma_1$ can be expressed as a graph of the same function as in the case of $\Gamma$ namely $x\rightarrow D_xd_X(0)$. The medial axis of the epigraph $d(x,m(0))$ after the translation by $(0,-1)$ has to coincide with $M_{C_{(0,1)}\Gamma}$, thus their intersections with the cone $\lbrace y\leq \alpha\|x\|\rbrace$ are also equal.

We have obtained $$C_0M_X\times\Rz\cap\lbrace y\leq \alpha\|x\|\rbrace= C_{(0,1)}M_\Gamma\cap\lbrace y\leq \alpha\|x\|\rbrace$$ and $$M_{m(0)}\times \Rz\cap\lbrace y\leq \alpha\|x\|\rbrace=M_{C_{(0,1)}\Gamma}\cap\lbrace y\leq \alpha\|x\|\rbrace.$$ Since, as we mentioned at the beginning, $M_{C_{(0,1)}\Gamma}\subset C_{(0,1)}M_\Gamma$ the assertion follows.
\end{proof}

As the following example shows, the equality between $C_a M_{X}$ and $M_{m(a)}$ cannot be expected in general.

\begin{ex}\label{m(0) nierowne stozkowi}
Let $X=\lbrace (x,y,z)\in \mathbb{R}^3|\; z^2=1, (x+y)(x-y)=0 \,\rbrace$, then there is
\begin{itemize}
    \item $M_X=\lbrace (x,y,z)\in\Rz^3| xy=0,x\neq y\rbrace\cup\lbrace(x,y,z)\in\Rz^3|z=0\rbrace$,
    \item $m(0)=\lbrace (0,0,1),(0,0,-1)\rbrace$,
    \item $M_{m(0)}=\lbrace (x,y,z)\in\Rz^3|z=0\rbrace$.
\end{itemize}

It is easy to check that indeed $M_{m(0)}$ is a proper subset of $C_0 M_X$.
\end{ex}
\begin{figure} 
    
        \includegraphics[width=0.5\textwidth]{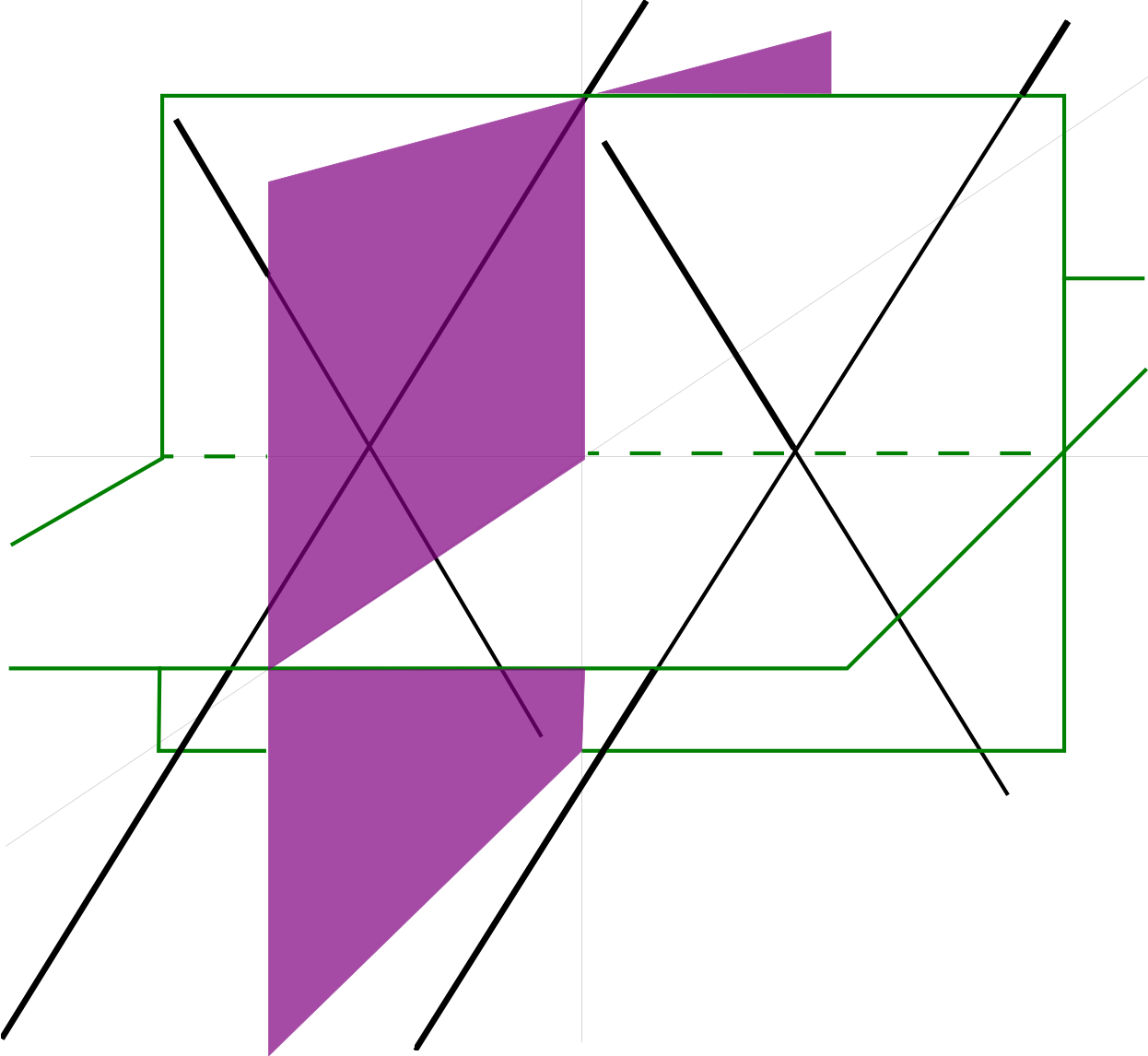}

    \caption{Example \ref{m(0) nierowne stozkowi}. Even though the medial axis of the double X consists of all off the visible surfaces, only the shaded purple one is a part of $M_{m(0)}$. }
\end{figure} 
Problems in the reconstruction of the whole tangent cone of $M_X$ basing solely on $m(a)$ come from sequences of points $x_\nu$ in the medial axis with $m(x_\nu)$ diverging to a singleton. Assuming no such sequence can be found, we can prove the following.

\begin{cor}\label{rownosc}
Assume that $0\in M_X$ for a closed definable $X\subset \Rz^n$. If there exists a neighbourhood of the origin $U$ and $r>0$ such that for any $a\in U\cap M_X$ there is $\text{diam } m(a)>r$, then $C_0M_X=M_{m(0)}$.
\end{cor}
\begin{proof}
Theorem \ref{Stozek} gives us one of the inclusions in question. To prove the other one, start by taking $v\in C_0M_X$. By the definition we can find sequences $\lbrace a_\nu\rbrace$ in $M_X$ and $\lbrace \lambda_\nu\rbrace$ in $\Rz_+$ such that
$$a_\nu\rightarrow 0,\, \lambda_\nu\rightarrow 0,\text{ and } a_\nu/\lambda_\nu \rightarrow v.$$

Take any convergent sequence of elements $m(a_\nu)\ni x_\nu\rightarrow x\in m(0)$, we will show that $x\in m_{m(0)}(v)$. Since the additional assumption on the diameter of $m(a_\nu)$ ensures $\text{diam } \limsup m(a_\nu)\geq r>0$ this will give $v\in M_{m(0)}$.

Consider $\mathbb{B}(a_\nu,d(a_\nu))\cap \mathbb{S}(0,d(0))$. For $\|a_\nu\|<d(0)$ it is a closed ball $B_\nu$ in $\mathbb{S}(0,d(0))$ centered at $a_\nu/\|a_\nu\|$. Moreover, $\mathbb{B}(a_\nu,d(a_\nu))\cap X= m(a_\nu)$ also ensures $B_\nu\cap m(0)=m(a_\nu) \cap m(0)$. Since $x_\nu\rightarrow x$, the sequence of balls $B_\nu$ converges to a certain closed ball $B$ centered at $v$ with $x$ on its boundary. Of course, the interior of $B$ has an empty intersection with $m(0)$, as for every $\varepsilon>0,\, \mathbb{B}(a_\nu, d(a_\nu) - \varepsilon) \cap X=\emptyset$,  which proves that $x$ is the closest point to $v$ in $m(0)$.

\end{proof}

Theorem \ref{Stozek} yields two immediate yet interesting corollaries.

\begin{cor}
\label{dim}
In the considered situation, $\dim_a M_X\geq \dim M_{m(a)}.$
\begin{proof}
From Theorem \ref{Stozek} the medial axis $M_{m(a)}$ is a subset of the tangent cone $C_a M_X$, thus its dimension is bounded by $\dim C_a M_X$. Since $M_X$ is definable, $\dim M_X$ is always greater than or equal to $\dim C_a M_X$ and the assertion follows.
\end{proof}
\end{cor}
The following result was known before (e.g. \cite{Cannarsa} Theorem~6.2 or \cite{DenkowskiOnPoints} Theorem~4.10  for the subset of points in the medial axis with $\dim m(a)=n-1$), Theorem~\ref{Stozek} yields a proof more natural in the medial axis category.
\begin{cor}\label{Full Sphere}
Point $a\in M_X$ is isolated in $M_X$ if and only if $m(a)$ is a full sphere.
\begin{proof}
Sufficiency of the condition is obvious.
For the proof of necessity, suppose that $m(a)$ is not a full sphere. It is easy to observe (for example, using the compactness of the sphere, the continuity of the distance function, and Theorem~\ref{Szkielet podzbioru sfery}), that its medial axis is a cone of dimension $\dim M_{m(a)}>0$. From Theorem~\ref{Stozek} we derive that $\dim C_aM_X>0$, hence $a$ cannot be isolated.
\end{proof}
\end{cor}

In other words, every eyelet in $m(a)$ enables an escape of $M_X$ in its general direction.

On the plane the situation is, as usual, simpler

\begin{thm}\label{plane_case}
For a closed definable $X\subset\mathbb{R}^2$ there is always $$C_aM_X=M_{m(a)}-a.$$
\end{thm}
\begin{proof}
Assume $a=0$. 
Of course, in view of Theorem~\ref{Stozek} the only strict inclusion possible is $M_{m(0)}\subsetneq C_0M_X$. Take then $v\in C_0M_X \backslash M_{m(0)}$ if one exists. Since $M_X$ is definable, by the Curve Selection Lemma there exists a continuous $\gamma: [0,1]\to M_X$ tangential to $v$ with $\gamma(0)=0$ and  with $0\notin \gamma((0,1])$. Moreover, as for every $x\in m(0)$ a segment $(0,x]$ does not intersect $M_X$, the image of $\gamma$ and the set $(0,1]\cdot m(0)$ must be disjoint. Take now $\nu\in M_{m(0)}\subset C_0M_X$ lying in the same connected component of $\mathbb{R}^2\backslash (\mathbb{R}_+\cdot m(0))$ as $v$ and denote by $\psi$ a curve in $M_X$ tangential to $\nu$. For any vector $w\in \mathbb{R}^2$ there is $\lim_{t\to 0^+}tw = 0$ and $\limsup_{t\to 0^+} m(tw)\subset m(0)$. Therefore, for any $w$ in a segment $(v,\nu)$ with an arbitrary choice of $w_t\in m(tw)$, starting from a certain $T>0$, a segment $[tw,w_t]$ has to intersect eventually either $\gamma$ or $\psi$, what ends in a contradiction.
\end{proof}

\begin{figure} 
    
        \includegraphics[width=0.5\textwidth]{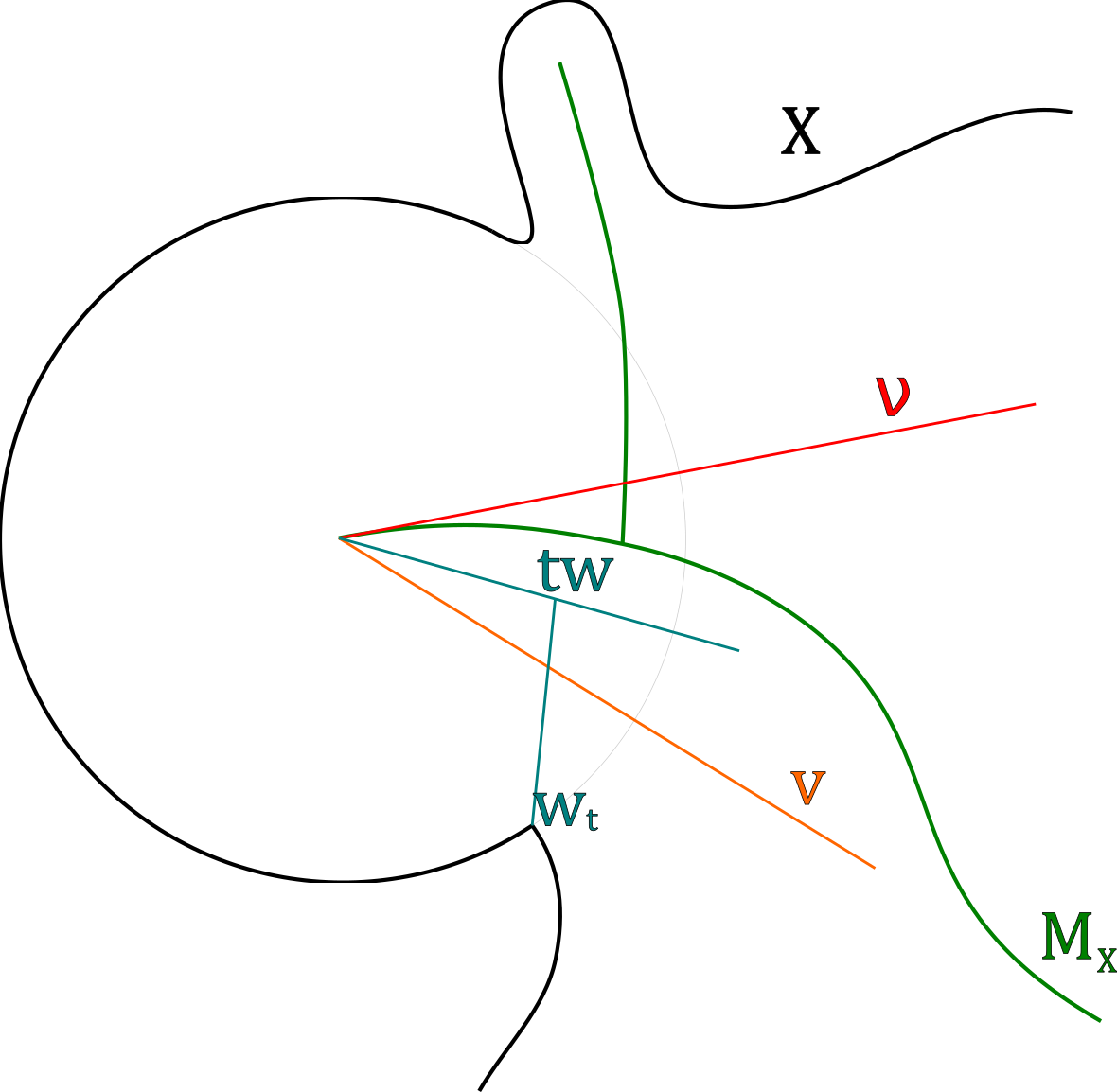}

    \caption{Theorem \ref{plane_case}. Starting from a certain $T>0$, a segment $[tw,w_t]$ intersects either $\gamma$ or $\psi$. }
\end{figure} 

\section{The dimension of the medial axis}

Our aim now is to use the established relation between the tangent cone of the medial axis at point $x\in M_X$ and the medial axis of $m(x)$, to describe the dimension of the medial axis in a more refined way than Corollary \ref{dim}. This result can be viewed as an answer to the hypothesis posted in \cite{DenkowskiOnPoints} or \cite{Erdos}.

Recall again that the multifunction $m(x)$ is definable if X is a definable set. By the definition it means that the graph $$\Gamma_m:=\lbrace (x,y)\in\Rz^n\times\Rz^n|y\in m(x)\rbrace$$ is definable and thus there exists a cylindrical definable cell decomposition ($cdcd$, cf. \cite{Coste})  $\lbrace D_1,\ldots, D_\alpha\rbrace$ of $\Rz^n\times\Rz^n$ adapted to $\Gamma_m\cap (M_X\times \Rz^n)$. Without much effort functions defining cells in $cdcd$s can be assumed to be of class $\mathcal{C}^k$ for an arbitrary $k\in\mathbb{N}$, although in this paper we will not have the need of such smoothness, thus just continuity of the functions is assumed. From the definition of $cdcd$, the collection of the projections of $D_i$ to the first $n$ coordinates forms a $cdcd$ of $\Rz^n$ adapted to $M_X$. Using this decomposition we will be able to prove further properties of the multifunction $m(x)$ and as a consequence we will obtain an explicit formula for the dimension of $M_X$. 

\begin{dfn}
Let $\mathcal{D}$ be a \textit{cdcd} of $\mathbb{R}^n\times\mathbb{R}^n$ adapted to $\Gamma_m$. Denote by $\mathcal{C}$ a \textit{cdcd} of $\mathbb{R}^n$ obtained by projections of cells in $\mathcal{D}$ onto the first $n$~coordinates. We call $x_0\in M_X$ an \textit{interior point of $M_X$ with respect to $\mathcal{D}$} if
there exists a neighbourhood $U_0$ of $x_0$ such that $U_0\cap C = U_0 \cap M_X$, where $C$ is the
unique cell in $\mathcal{C}$ containing $x_0$.
\end{dfn}

\begin{rem}
Mind that since every cell of a $cdcd$ is of pure dimension, the condition $U_0\cap C=U_0\cap M_X$ implies $\dim_{x_0}M_X = \dim C$.
\end{rem}

\begin{prop}\label{ciaglosc}
Let $\mathcal{D}$ be a $cdcd$ of $\mathbb{R}^n\times\mathbb{R}^n$ adapted to $\Gamma_m$ and $M_X\times \Rz^n$. Assume that $x_0 \in M_X$ is an interior point of $M_X$ w.r.t. $\mathcal{D}$. Then the multifunction $m|_{M_X}$
is continuous at $x_0$.
\end{prop}

\begin{proof}
It is known \cite{DenkowskiOnPoints}, \cite{BirbrairDenkowski} that the multifunction $m(x)$ is upper se\-mi-con\-tin\-uous along $M_X$: for any $x_0\in \Rz^n$ we have $\limsup_{M_X\ni x\rightarrow x_0}m(x)\subset m(x_0)$. To prove the continuity of $m(x)$, we need to show that in the considered case $m(x_0)$ is a subset of the lower Kuratowski limit as well. Explicitly, we need to show that
$$\forall y\in m(x_0),\ \forall U\ni y,\ \exists V\ni x_0\colon \forall x\in V\cap M_X,\ m(x)\cap U\neq \emptyset,$$
where $U,V$ are open sets.

Take $y=(y_1,\ldots,y_n)\in m(x_0)$ and a neighbourhood $U_1\times\ldots\times U_n$ of $y$. Denote by $D$ the cell in $\mathcal{D}$ containing $(x_0,y)$ and by $C$ its projection on $\mathbb{R}^n$. We will show the %lower outer? 
continuity of $\pi_i\circ m(x)$ where $\pi_i$ is a natural projection on the first $i$ coordinates. Of course, $\pi_n\circ m(x)=m(x)$ and the assertion will follow.

For the first coordinate, observe that $C_1:=(id_{\mathbb{R}^n}\times\pi_1) (D)$ is either a graph of a continuous function or a band between two such functions defined over $C$. In the first case, we can easily find a neighbourhood $V$ of $x_0$ in $M_X$ such that $\pi_1\circ m(V)\subset U_1$ which implies $\pi_1\circ m(V)\cap U_1\neq \emptyset$. In the case $C_1$ is a band between two functions $f_-,f_+$ there must be $f_-(x)<y_1<f_+(x)$. By continuity these inequalities must hold in a certain neighbourhood $V$ of $x_0$, resulting in $$\sup_{x\in V}f_-(x)\leq y_1\leq \inf_{x\in V}f_+(x).$$ Clearly $\pi_1\circ m(V)\cap U_1\neq \emptyset$.

Assume now the composition $\pi_k \circ m(x)$ to be continuous for $k<n$. Again, the cell $C_{k+1}:=(id_{\mathbb{R}^n}\times\pi_{k+1}) (D)$ is either a graph of a continuous function $f_{k+1}$ or a band between two such functions, this time defined over $C_k$. In the first case, we can find a neighbourhood $V=V_0\times\ldots\times V_{k}$ of $(x,y_1,\ldots,y_k)$ such that $f_{k+1}(V)\subset U_{k+1}$. As $\pi_k\circ m(x)$ is continuous, we can ensure $(x,y'_1,\ldots,y'_k)\in V_0\times V_1\times\ldots\times V_k$ just by shrinking $V_0$, by shrinking it even further we can also ensure $(y'_1,\ldots,y'_k)\in U_1\times\ldots\times U_k$.
By doing that, we obtain $\pi_{k+1}\circ m(V_0)\subset U_1\times\ldots\times U_{k+1}$.
The case of $C_{k+1}$ being a band follows in the same manner. 

As was mentioned earlier, as a consequence we obtain the continuity of $m(x)$ at every point of $C$, and in particular at $x_0$.
\end{proof}

The proof of Theorem~\ref{ciaglosc} gives a soothing correlated result on the continuity of $m(x)$. In an o-minimal setting a set of discontinuities of $m(x)$ is always nowhere dense and is equal to the medial axis (see \cite{Fremlin} for an counter-example outside of o-minimal setting). Luckily enough, the restriction of $m(x)$ to the medial axis exhibits an analogous behaviour. It is still continuous outside of a subset nowhere dense in the induced topology. Consequently, with $\mathbb{R}^n$ being locally compact, we have

\begin{cor}\label{gestosc}
The set $$\lbrace a\in M_X| C_a M_X = M_{m(a)}\rbrace$$ is open and dense in $M_X$.
\end{cor}
\begin{proof}
For any $cdcd$ adapted to a definable set $M_X$ points admitting a neighbourhood $U$ such that $U\cap C=U\cap M_X$ form an open and dense subset of $M_X$. For any such point, we can find a relatively compact neighbourhood $V$ included in its cell.  
The theorem now follows Proposition \ref{ciaglosc} and Corollary \ref{rownosc}, as $diam\, m(x)$ is positive and continuous on $V$. 
\end{proof}

The main result of this paper, settling the question about the dimension of the medial axis in the definable case, is the following.

\begin{thm}\label{wymiar}
Let $\mathcal{D}$ be a $cdcd$ of $\mathbb{R}^n\times\mathbb{R}^n$ adapted to $\Gamma_m\cap (M_X\times \Rz^n)$. If $x_0 \in M_X$ is an interior point of $M_X$ w.r.t. $\mathcal{D}$ then an equality occurs $$\dim_{x_0} M_X + \dim m(x_0)=n-1.$$
\end{thm}
\begin{proof}
Without loss of generality, we can assume $x_0=0$.
 
We will prove the theorem by induction with respect to the dimension of $C$ - the cell containing $x_0$.
 
The case $\dim C=0$ is already proved in Corollary \ref{Full Sphere}

Assume now that the dimension of $C$ is equal $k$, and the theorem holds for any cell $C'$ of dimension smaller than $k$. 
Since $0$ satisfies the assumptions of Proposition \ref{ciaglosc}, we can think about $C_0 M_X$ as a medial axis of $m(0)$, then clearly $$\dim_0 M_X=\dim C\geq \dim C_0C=\dim M_{m(0)}.$$ 
Corollary~\ref{gestosc} states, that we can find $0\neq v\in M_{m(0)}$ %with $\dim_vM_{m(0)}=\dim M_{m(0)}$
for which the tangent cone to $C_0 M_X $ at $v$ depends only on points in $m(0)$ that are lying closest to $v$ (and are collected in a set $m_{m(0)}(v)$). Therefore, for such $v$, after a suitable translation, there is again $C_vM_{m(0)}=M_{m_{m(0)}(v)}$ and $$\dim M_{m(0)}\geq \dim C_v M_{m(0)}=\dim M_{m(0)}.$$ %To prove the theorem we need to show now that $\dim m(0)\geq n- \dim M_{m_{m(0)}(v)}$ 

The set $m_{m(0)}(v)$ is a subset of both $\mathbb{S}(0,d(0,X))$ and $\mathbb{S}(v,d(v,m(0)))$. Denote by $L$ an unique $n-1$ dimensional affine subspace of $\Rz^n$ containing an intersection of the mentioned spheres. It is immediate that the subspace $L$ can be written as $v^\perp+\alpha v$ for a certain $\alpha\in \Rz$ and that $m_{m(0)}(v)$ is an subset of $L$. Therefore, the medial axis of $m_{m(0)}(v)$ in $\mathbb{R}^n$ is a Minkowski sum of $v\Rz$ and the medial axis of $L\cap m(0)$ computed in $L$ (denoted by $M^L_{m_{m(0)}(v)}$).

\begin{figure} 
    
        \includegraphics[width=0.4\textwidth]{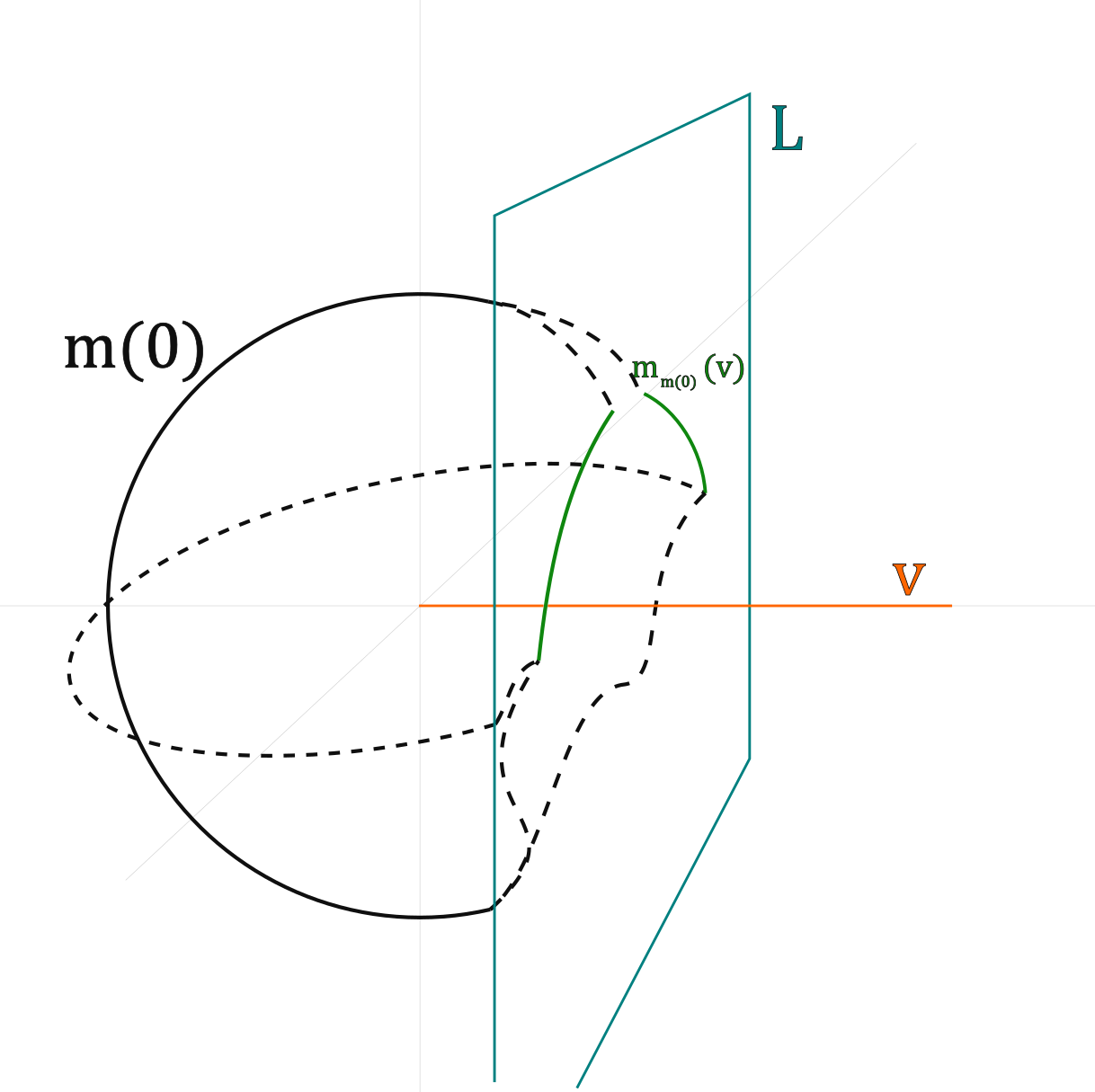}

    \caption{Theorem \ref{wymiar}. The set $m_{m(0)}(v)$ is a subset of an affine space orthognal to $v$. }
\end{figure} 

This means that $\dim M_{m_{m(0)}(v)}=\dim M^L_{m_{m(0)}(v)}+1$ and in particular we have $\dim M^L_{m_{m(0)}(v)} < k$. Due to the last inequality, every cell in the $cdcd$ of $L$ adapted to $M^L_{m_{m(0)}(v)}$ has a dimension bounded by $k-1$ and so, by the induction hypothesis, for a generic $x\in M^L_{m_{m(0)}(v)}$ there is
$$\dim  m_{m(0)\cap L}(x)=\dim L-1 - \dim M^L_{m_{m(0)}(v)}.$$

A set $m_{m(0)\cap L}(x)$ is a subset of $m(0)$, so its dimension cannot exceed $\dim m(0)$, whereas on the other side of the equality we have obtained in fact $n-1-\dim M_{m_{m(0)}(v)}$, which by our choice is greater than or equal $n-1-\dim M_X$.

The opposite inequality is way simpler to prove. It suffices to observe that $U\cap C = U\cap M_X$ is a guarantee that the dimension of $m(x)$ is constant in $U$. Consequently  \cite{DenkowskiOnPoints} Theorem~4.13 ensures that the sum $\dim_0 M_X+ \dim m(0)$ cannot exceed $n-1$. 

Finally, we obtain the desired $\dim_0 M_X+ \dim m(0) = n-1$.

\end{proof}

The formula for a generic point allows to describe the dimension at any point of $M_X$. This strengthening of the results from \cite{Cannarsa},\cite{DenkowskiOnPoints}, and \cite{Erdos2} solves the problem for sets definable in an o-minimal setting. 

\begin{thm}\label{ogolny}
For any point $a\in M_X$ there is $$\dim_a M_X+\min \lbrace k\,|\,a\in \overline{M^k}\rbrace=n-1$$
Where $M^k=\lbrace a\in M_X|\dim m(a)=k\rbrace$.
\end{thm}
\begin{proof}

We will prove that 
$$\min \lbrace k|a\in \overline{M^k}\rbrace = n-1-\alpha \iff \dim_a M_X=\alpha$$
holds true for any $\alpha\in\mathbb{N}$.

For $\alpha =0$, one of the implications is exactly what we know from Corollary $\ref{Full Sphere}$. The opposite one is given by \cite{DenkowskiOnPoints} Theorem~4.10. 

Now assume the claim to be true whenever $\alpha<\alpha_0$. Mind that \cite{DenkowskiOnPoints} Theorem~4.13 states that for any $x\in M_X$ 
$$\dim m(x)+ \dim M^{\dim m(x)}\leq n-1,$$ thus it is easy to observe that
$$\min \lbrace k|a\in \overline{M^k}\rbrace = n-1-\alpha_0 \Rightarrow \dim_a M_X=\alpha_0.$$

It remains to prove the opposite implication. Take any $a\in M_X$ with $\dim_a M_X=\alpha_0$.  
The claim for $\alpha<\alpha_0$ allows us to analyse just points in the vicinity of $a$ where the local dimension of $M_X$ is equal to $\alpha_0$. Furthermore, it is only needed to be showed $\dim m(a)\geq n-1-\alpha_0$.

Assume otherwise: $\dim m(a)<n-1-\alpha_0$. Surely the dimension $\dim m_{m(a)}(v)$ is smaller than $n-1-\alpha_0$ for any $v\in M_{m(a)}$ as well. Moreover, thanks to Corollary \ref{gestosc} we can find a point $v\in M_{m(a)}$ with $\dim_v M_{m(a)}+\dim m_{m(a)}(v)=n-1$. Now $$\alpha_0<n-1-\dim m_{m(a)}(v)=\dim_v M_{m(a)}\leq \dim_a M_X=\alpha_0$$
gives the contradiction sought for.

\end{proof}

%\begin{wrapfigure}{r}{0.25\textwidth}
%    \centering
%    \includegraphics[width=0.25\textwidth]{Graphics/EX37.png}
%    \caption{The wristwatch.}
%\end{wrapfigure}

Let us remark that in order to describe the dimension of the medial axis $M_X$ at a given point $a_0\in M_X$, it is indeed necessary to find the minimum of the dimensions of $m(a)$ for $a$ in a sufficiently small neighbourhood $U$ of $a_0$. 
  
\begin{ex}[Wristwatch]\label{Wristwatch}
Let $X\subset\Rz^2$ be the boundary of a closed set $\mathbb{B}(0,2)\cup((-1,1)\times\Rz)$. Then 
\begin{multline*}
 \dim_0 M_X  +\dim m(0) =\\ 
 =\dim (\lbrace 0\rbrace\times \Rz) +\dim \lbrace x^2+y^2=2,|x|\geq1\rbrace=2.
\end{multline*}
\end{ex}
\begin{figure} 
    
        \includegraphics[width=0.28\textwidth]{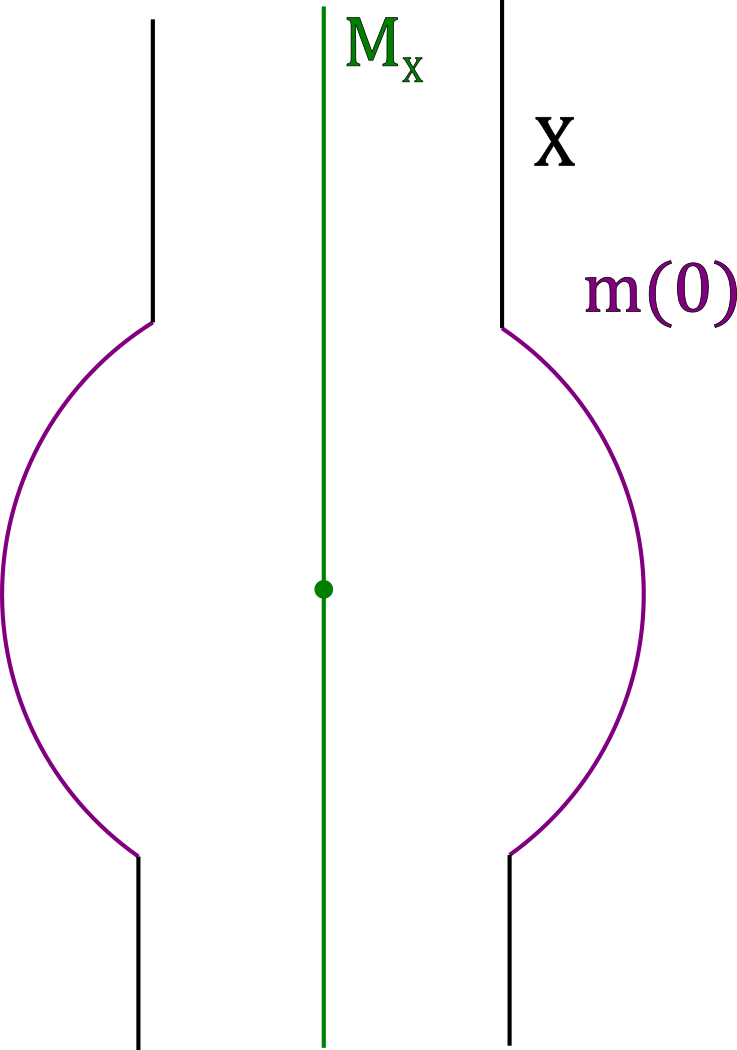}

    \caption{Example \ref{Wristwatch}.  }
\end{figure} 
Finally, we give a global formula for the dimension of a medial axis.
\begin{cor} 
For a closed definable set $X\subset \Rz^n$
$$\dim M_X=n-1-\min_{a\in M_X} \dim m(a).$$
\end{cor}
\begin{proof}
Obvious from Theorem \ref{ogolny}.
\end{proof}

\begin{rem}
In case $X$ is a collection of isolated points, its medial axis $M_X$ is exactly a conflict set of the singletons included in $X$. The dimension of $m(a)$ is equal to zero for any point $a\in\mathbb{R}^n$ and the global formula indeed gives $\dim M_X=n-1$, as predicted by the theory of conflict sets \cite{BirbrairSiersma}. Conversely, whenever for a given $a\in M_X$ the set $m(a)$ is not connected, the conflict set theory assures that the dimension $\dim_a M_X = n-1$. General formula for the dimension of the medial axis concludes then that in every  neighbourhood of $a$ there must exist a point $b$ of $M_X$ with $m(b)$ finite.
 
\end{rem}

\section{A frontier of a medial axis}

In \cite{Miura} T.Miura proposed a characterisation of the medial axis boundary for hypersurfaces. Unfortunately, the introduced notion did not escape flaws. Foremost, it does not recognise the studied side of the hypersurface, which may result in misleading data. In this paper we provide an improved definition, resulting in a generalisation of the claims from \cite{Miura}, proved with more straightforward reasoning.

To take notice of the direction of open balls used to study $X$, we define as follows. Take any point $a\in X$ and write $V_a:=N_aX\cap\mathbb{S}$ to be a set of \textit{directions normal} to $X$ at $a$. Then denote a \textit{limiting set of normal directions} by $$\widetilde{V}_a:=\limsup_{a_\nu\to a}V_{a_\nu}.$$

\begin{rem}
If $a\in X$ is a point of $\mathcal{C}^1$ smoothness, then of course the tangent spaces, and, what follows, also the normal spaces are continuous at $a$. Therefore, in such case, the limiting set of normal directions is just the set of normal directions.\end{rem}
Utilising the introduced sets, define for a point $a\in X$ as follows.

\begin{dfn}
For $v\in V_a$ we define a \textit{directional reaching radius} 
$$r_v(a):=\sup \lbrace t\geq 0|\,a\in m(a+tv)\rbrace.$$

Then for $v\in \widetilde{V}_a$ we define a \textit{limiting directional reaching radius}
$$\Tilde{r}_v(a):=\liminf\limits_{X\ni x\to a, V_x\ni v_x\to v\in \widetilde{V}_a}r_{v_x}(x),$$

and finally a \textit{reaching radius} at $a$
$$r(a)=\inf_{v\in \widetilde{V}_a}\tilde{r}_v(a).$$
\end{dfn}

Recall that for any $v\in N_aX$ a point $a+vr_v(a)$ is a center of a maximal (in a sense of inclusion) ball contained in $\mathbb{R}^n\backslash X$. A set of centers of maximal balls are gathered in a set called \textit{central set of X}. It is known (cf \cite{BirbrairDenkowski}) that a central set of a closed set lies between the medial axis of the set and its closure. 

At first glance, the infimum in the definition of the reaching radius might seem dubious for points with an empty limiting set of normal directions. By the definition of the infimum, we are inclined to post for any such point an infinite value of the reaching radius. Luckily, in the o-minimal geometry, the points with an empty limiting set of normal directions prove to be precisely the interior points of a given set. Clearly, any interior point has an empty limiting set of normal directions. On the other hand, every point on the boundary of a certain set $X$ can be reached by the regular part of the boundary of $X$. This is a consequence of the definability of the boundary and the nowhere denseness of the subset of singularities. Points of the regular part of the boundary of $X$ clearly have at least one normal direction, thus due to the compactness of the sphere, the limiting set of normal directions for the boundary points cannot be empty. 

\begin{rem}
As is easily seen from their definitions, both $\Tilde{r}_{v}(a)$, and $r(a)$ are lower semi-continuous functions. Moreover for $v\in \widetilde{V}_a\backslash V_a$ the limiting directional reaching radius is equal to zero.
\end{rem}

The backbone of the just defined reaching radius lies in the same place as the reaching radius introduced in \cite{BirbrairDenkowski} as 
$$\Dot{r}(a)=\begin{cases}
r'(a),& a\in Reg_2X,\\
\min\lbrace r'(a),\liminf\limits_{X\backslash\lbrace a\rbrace\ni x\to a} r'(a)\rbrace, & a\in Sng_2X  
\end{cases},$$
where
$$r'(a)=\inf_{v\in V_a}r_v(a)$$
is called \textit{weak reaching radius}.
One can perceive the difference between them as a sort of an order of taking limits problem. This paper emphasises the directional distance of the medial axis from a point first, rather than the distance from the point as Birbrair and Denowski did. Since the Birbrair-Denkowski reaching radius proved to be successful in describing $\overline{M_X}\cap X$, it would be desirable to achieve at least a type of correspondence between these two notions. Fortunately, as we will see in Theorem~\ref{Denkowski=mine}, the final results of both constructions are equal. 
For the sake of the next preparatory proposition, recall that \textit{the normal set} at $a\in X$ defined in \cite{BirbrairDenkowski} as
$$\mathcal{N}(a):=\lbrace x\in\mathbb{R}|\, a\in m(x)\rbrace $$
is always convex and closed. Moreover, for a subset of a unit sphere $A$ denote by $conv_\mathbb{S}(A)$ its convex hull in the spherical norm.

\begin{prop}
For any $a\in X$ a function $$\rho: V_a\ni v\rightarrow \rho(v) = r_v(a)\in [0,+\infty]$$ is upper semi-continuous on $V_a$.
Furthermore, it is continuous at $v\in V_a$ if there exist $r,\varepsilon>0$ such that $\mathbb{B}(v,r)\cap V_a=conv_\mathbb{S} (\mathbb{S}(v,r)\cap V_a)$ and $v\notin \overline{\rho^{-1}([0,\varepsilon])}$.
\end{prop}
\begin{proof}
Observe firstly, that for an empty $V_a$ the function $\rho$ is continuous by the definition. Therefore, assume for the rest of the proof that $V_a$ is nonempty.

To prove the upper semi-continuity take $v_0\in V_a$ and any sequence of $v_\nu\rightarrow v_0$. Then, for any $r< \rho(v_\nu)$ a point $(a+rv_\nu)$ is contained in $\mathcal{N}(a)$. Therefore, from the closedness of $\mathcal{N}(a)$, a point $(a+rv)$ must lie in $\mathcal{N}(a)$ for every $r<\limsup_{v\to v_0} \rho(v)$ . Moreover, $\mathcal{N}(a)$ is convex, so the whole segment $[a,a+rv_0]$ must be a subset of $\mathcal{N}(a)$ as well. This means that $\rho(v_0)\geq \limsup_{v\to v_0}\rho(v)$.

For the sake of lower semi-continuity assume that $v\in V_a\backslash\overline{\rho^{-1}([0,\varepsilon])}$ for certain  $\varepsilon>0$. Now we can find $r>0$ small enough that $\rho(w)>\varepsilon$ for $w\in \mathbb{S}(v,r)\cap \mathcal{N}(a)$. Since $\mathbb{B}(v,r)\cap\mathcal{N}(a)=conv \mathbb{S}(v,r)\cap \mathcal{N}(a)$, by the convexity of $\mathcal{N}(a)$, the value of $\rho(x)$ for $x=tv+(1-t)w$ is bounded from below by $t\alpha+(1-t)\rho(w)$, for any $\alpha <\rho(v)$. Therefore, at $v$ the function $\rho$ must be lower semi-continuous . 
\end{proof}
Mind that even though the normal set $\mathcal{N}(a)$ is convex for any $a\in X$ it does not necessarily mean that an $r>0$ satisfying $\mathbb{B}(v,r)\cap V_a=conv_\mathbb{S} (\mathbb{S}(v,r)\cap V_a)$ for every $v\in V_a$ exists. Indeed, only an inclusion from the right to the left is automatic. Take for an example $X=\lbrace z=\sqrt{x^2+y^2}\rbrace$. We can see that $V_0=\lbrace z\leq \sqrt{x^2+y^2}\rbrace\cap\mathbb{S}$, thus for $v=(1,0,-1)$ and all $r>0$ there is $\mathbb{B}(v,r)\cap V_0\supsetneq conv_\mathbb{S} (\mathbb{S}(v,r)\cap V_a)$.

\begin{cor}
The function $\rho: V_a\ni v\rightarrow \rho(v) = r_v(a)\in \mathbb{R}$ is continuous for any $a\in Reg_2 X$.
\end{cor}
\begin{proof}
Corollary follows easily the fact that $\overline{M_X}\cap Reg_2X=\emptyset$. Clearly, there must exist such $\varepsilon>0$ that $\rho(v)>\varepsilon$ for every $v\in V_a$. Furthermore, $V_a$ is just an intersection of a unit sphere with a normal space $N_aX$. Therefore, it is isomorphic to $\mathbb{S}^{n-\dim X}$, thus $$\mathbb{B}(v,r)\cap V_a=conv_\mathbb{S}(\mathbb{S}(v,r)\cap V_a)$$ for any $r<2$.
\end{proof}

Whenever the limiting directional reaching radius is positive, it can be seen as a limiting directional reaching radius transported from a $\mathcal{C}^1$ submanifold formed in a certain open set by $d^{-1}(\varepsilon)$.

\begin{lem}\label{szkielet_dla_epsilon}
For $X$ a closed subset of $\mathbb{R}^n$ and $\varepsilon>0$, denote $$X^\varepsilon:=\lbrace x\in \mathbb{R}^n|d(x,X)\leq\varepsilon\rbrace.$$ Then $x\in M_{X^\varepsilon}$ if and only if $x\in M_X$ and $d(x,X)>\varepsilon$. 
\end{lem}
\begin{proof}
Surely $x\in M_{X^\varepsilon}$ implies $d(x,X)>\varepsilon$, otherwise $x$ would be a point of $X^\varepsilon$. Furthermore, for any point $x\in\mathbb{R}^n$ with $d(x,X)>\varepsilon$ there is $$d(x,X)=d(x,X^\varepsilon)+\varepsilon.$$

Now, for any point $x\notin X^\varepsilon$ the set $m_{X^\varepsilon}(x)$ is just $m_X(x)$ scaled by a homothety of ratio $\frac{d(x,X)-\varepsilon}{d(x,X)}$ centered at $x$. Therefore, $m_{X^\varepsilon}(x)$ is a singleton if and only if $m_X(x)$ is one as well.
\end{proof}

\begin{prop}\label{radius_for_epsilon}
Take $a\in X$ a point of a closed subset of $\mathbb{R}^n$, $v\in V_a$, and $\varepsilon>0$. Denote by $\tilde{r}$ and $\tilde{r}^\varepsilon$ the limiting reaching radius for $X$ and $X^\varepsilon$ respectively. Then $\tilde{r}_v(a)=\tilde{r}^\varepsilon_v(a+\varepsilon v)+\varepsilon$ whenever $\tilde{r}_v(a)>\varepsilon$. 
\end{prop}
\begin{proof}

Since $\tilde{r}_v(a)>\varepsilon$, there exists $U$ a neighbourhood of $(a,v)$ in $VX:=\lbrace (x, v)|\, x\in X, v\in V_x\rbrace$ such that for any $(x,v_x)\in U$ there is $r_{v_x}(x)>\epsilon$. It means that for $a^\varepsilon:=(a+\varepsilon v)$ there exists a neighbourhood $W$ in $\mathbb{R}^n$ such that $\Gamma:=d^{-1}(\varepsilon)\cap W$ is a $\mathcal{C}^1$-smooth manifold. Moreover a series of equalities: $$N_{a^\varepsilon}\Gamma=(T_{a^\varepsilon}\Gamma)^\perp=(\nabla d)(a^\varepsilon)\cdot\mathbb{R}=\frac{a^\varepsilon - m(a^\varepsilon)}{\|a^\varepsilon - m(a^\varepsilon)\|}\cdot\mathbb{R} = (a^\varepsilon  - a)\cdot\mathbb{R}=v\mathbb{R},$$
proves that $v$ is a normal vector to $\Gamma$ at $a^\varepsilon$. Therefore, it is indeed possible to calculate $\tilde{r}_v^\varepsilon (a^\varepsilon)$.

According to Lemma~\ref{szkielet_dla_epsilon} the medial axes of $X$ and $X^\varepsilon$ coincide in $\mathbb{R}^n\backslash X^\varepsilon$. Therefore, for $(x,v_x)\in U$ from 
$x+r_{v_x}(x)v_x\in \overline{M_X\cap (x+v_x\mathbb{R})}$ we can derive easily $$x+r_{v_x}(x)v_x\in \overline{M_{X^\varepsilon}\cap (x+v_x\mathbb{R})} \text{ and } ([\varepsilon,r_{v_x}(x))\cdot v_x+x)\cap M_{X^\varepsilon}=\emptyset.$$
Thus, $r_{v_x}(x)=r_{v_x}^\varepsilon(x^\varepsilon)+\varepsilon$, where $x^\varepsilon:=x+\varepsilon v_x$, and $r^\varepsilon$ denotes the directional radius for $\Gamma$ (an explanation behind $v_x\in V_{x^\varepsilon}$ is the same as for $a$ in the first part of the proof). Furthermore, a sequence of  $(x_\nu,v_\nu)\in VX$ converges to $(a,v)$ if and only if a sequence $(x^\varepsilon_\nu,v_\nu)\in V\Gamma$ converges to $(a^\varepsilon,v)$. Therefore, the appropriate limits in the definition of the limiting reaching radius are equal.
\end{proof}

The main idea of the limiting directional reaching radius is to provide a suitable object for the generalisation of the results from \cite{Miura}. Indeed, the limiting directional reaching radius can be utilised to describe the frontier of the medial axis for a broader class of sets. Mind here, that in the contrast to the results in previous sections the following ones do not assume definability of a set $X$. 
 
\begin{thm}\label{Miurat}
Let $X$ be a closed subset of $\mathbb{R}^n$. Pick $x\in \mathbb{R}^n\backslash (X\cup M_X)$ and write $m(x)=\lbrace a\rbrace$, $v=\frac{x-a}{\|x-a\|}$. Then for $x\in\overline{M_X}$ there is $d(x)\geq \Tilde{r}_{v}(a)$. If additionally $\Tilde{r}_v(a)>0$, then $d(x)\geq \Tilde{r}_{v}(a)$ implies $x\in\overline{M_X}$.
\end{thm}
\begin{proof}
Assume that $x\in \overline{M_X}\backslash M_X$ and take a sequence of points $M_X\ni x_\nu\rightarrow x$. Since the multifunction $m(x)$ is upper semi-continuous, for an arbitrary choice of $a_\nu\in m(x_\nu)$ we also have $a_\nu\rightarrow a$. This of course means that $x_\nu-a_\nu\rightarrow x-a$. Taking $v_\nu=\frac{x_\nu-a_\nu}{\|x_\nu-a_\nu\|}$ we obtain by calculating the limiting directional reaching radius $$\Tilde{r}_{v}(a)\leq \liminf_{\nu\to\infty} r_{v_\nu}(a_\nu)=\lim_{\nu\to\infty} d(x_\nu)=d(x).$$

We will first prove the remaining part of the theorem with an additional assumption that $X$ is a $\mathcal{C}^1$ smooth submanifold in the neighbourhood of $a$. Assume accordingly that $\Tilde{r}_v(a)>0$ and $x\notin \overline{M_X}$. We will show that $d(x)<\Tilde{r}_v(a)$. 

At the very beginning, let us recall that outside of $\overline{M_X}\cup X$ the function $d(x)$ is of $\mathcal{C}^1$ class. Therefore, we can find $\varepsilon>0$ small enough that $m(x+\varepsilon v)=m(x)$ and a neighbourhood $U$ of $x^\varepsilon:=(x+\varepsilon v)$ such that $\Gamma:= d^{-1}(d(x^\varepsilon))\cap U$ is a $\mathcal{C}^1$ hypersurface disjoint from $M_X$. 

Now let us denote by $\Gamma'$ the intersection of $\Gamma$ and $(T_aX+\mathbb{R}v)$ translated by the vector $a$. The intersection is transversal as $v=\nabla d(x^\varepsilon)$, so $\Gamma'$ is a ($\dim X$)-dimensional $\mathcal{C}^1$ submanifold of $\mathbb{R}^n$. Mind that in particular tangent spaces to $X$ at $a$ and $\Gamma'$ at $x^\varepsilon $ are equal.

We claim that there exists an open neighbourhood $U'$ of $x^\varepsilon$ such that $m|U'\cap \Gamma'$ is an injection. Suppose otherwise. Then there exists a sequence of pairs of distinct points $x_\nu, y_\nu\in\Gamma'$ converging to $x^\varepsilon$ such that $m(x_\nu)=m(y_\nu)$. Since a multifunction $m$ is univalued in $U$, we can write
\begin{multline*}
    \frac{x_\nu-y_\nu}{\|x_\nu-y_\nu\|}=\\
    =\frac{1}{\|x_\nu-y_\nu\|}\left[ m(x_\nu)-d(x^\varepsilon)\nabla d(x_\nu)-(m(y_\nu)-d(x^\varepsilon)\nabla d(y_\nu)) \right].
\end{multline*}
Now, since $\Gamma'$ is $\mathcal{C}^1$ smooth left side of the equation tends to a vector in $T_{x^\varepsilon}\Gamma'=T_aX$ as $\nu\to\infty$ (cf. \cite{BigolinGrecko}). On the other hand the square bracket on the right side represents a difference of two vectors in $N_{m(x_\nu)}X$ which by the $\mathcal{C}^1$ smoothness of $X$ must tend to a vector in $N_aX$. This is a contradiction as the limit cannot be equal to zero. Therefore, the claim is proved.

Now, Brouwer Domain Invariance theorem asserts that $m|U'\cap\Gamma'$ is a homeomorphism, thus $m(U'\cap\Gamma')$ is an open neighbourhood of $a$ in $X$. Moreover, for $b\in m(U'\cap\Gamma')$ we have found the normal vectors $\eta_b$ such that $r_{\eta_b}(b)>d(x^\varepsilon)=d(x)+\varepsilon$ and $\eta_b\to v\,(b\to a)$. What is more, since $\Tilde{r}_v(a)>0$, all directional radii are continuous in a neighbourhood of $(a,v)$. Since $U'\cap\overline{M_X}=\emptyset$ this means that $d(x)<\Tilde{r}_v(a)$.

For $a\notin Reg_1 X$ observe, that for a positive $\varepsilon<\Tilde{r}_v(a)$ there exists a neighbourhood of $(a+\varepsilon v)$ such that $d^{-1}(\varepsilon)$ is a $\mathcal{C}^1$ submanifold of $\mathbb{R}^n$. In such a case the distance $d(x,X)$ equals $d(x,d^{-1}(\varepsilon))+\varepsilon$, and the medial axis $M_X$ coincides with $M_{d^{-1}(\varepsilon)}$ in a certain neighbourhood of $x$.
Moreover, for $(a_\nu,v_\nu)$ sufficiently close to $(a,v)$ the directional reaching radius $r_{v_\nu}(a_\nu)$ calculated for $X$ equals the directional reaching radius  $r_{v_\nu}(a_\nu+\varepsilon v_\nu)+\varepsilon$ calculated for $d^{-1}(\varepsilon)$.
Therefore, we can apply the result for $\mathcal{C}^1$ submanifolds to  $d^{-1}(\varepsilon)$ to obtain the assertion.
\end{proof}

In comparison to Miura's results, the main asset of the reaching radius-based approach is a lack of need to assume neither nonspreading normal cones for $X$, nor the graph structure of $X$. This gives a significantly broader application potential.

\begin{figure}[h]
    \centering

    \begin{subfigure}{0.48\textwidth}
        \includegraphics[width=\textwidth]{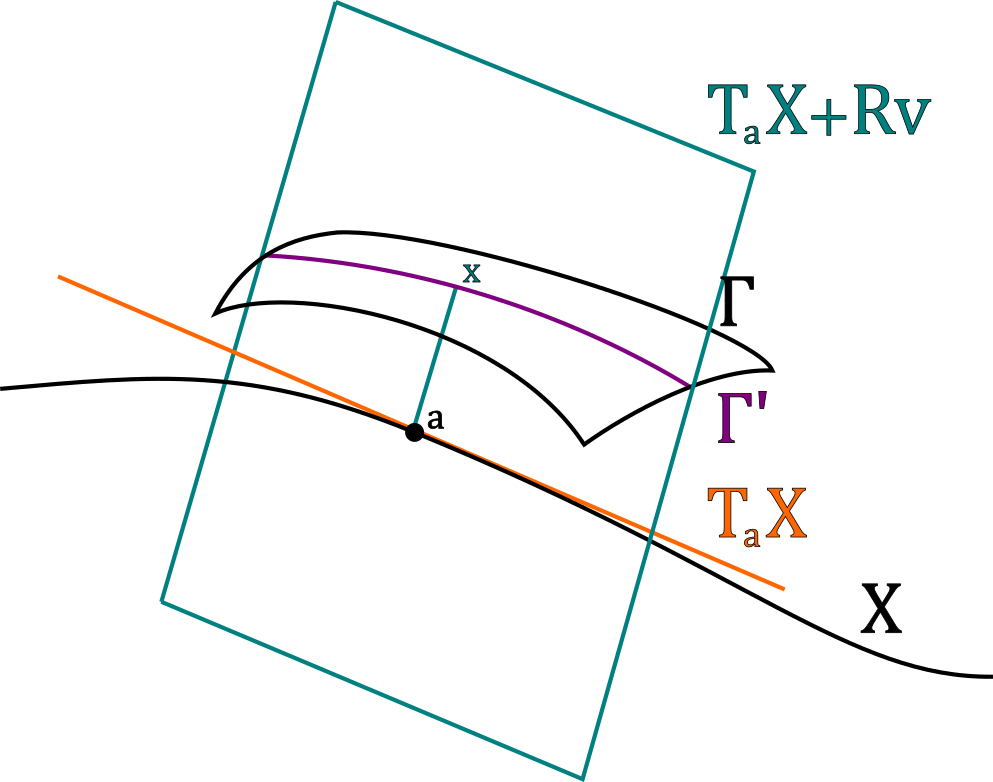}
        \caption{}
    \end{subfigure}  
    \begin{subfigure}{0.48\textwidth}
        \includegraphics[width=\textwidth]{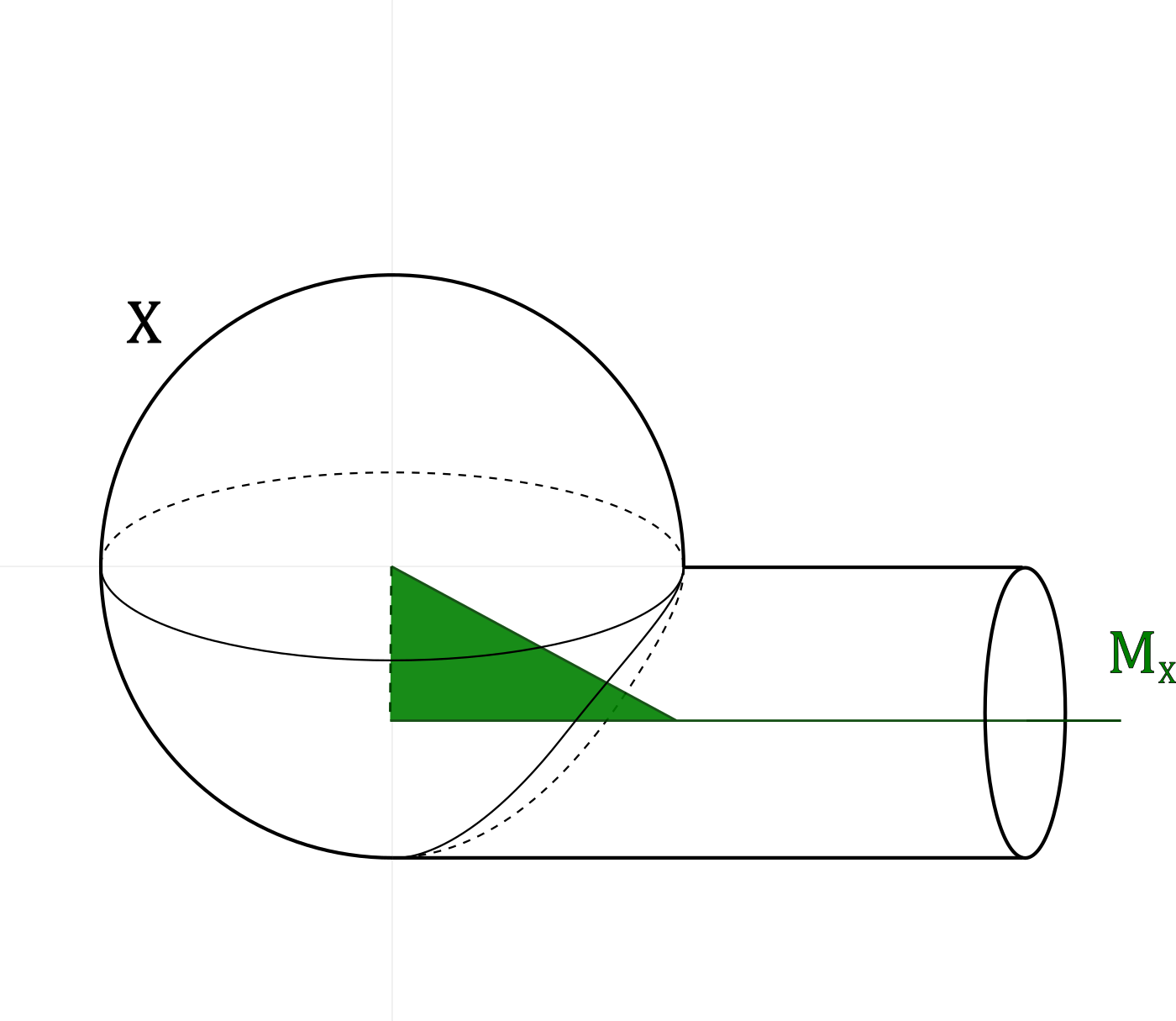}
        \caption{}
    \end{subfigure}
    \caption{(A) Theorem \ref{Miurat}. The intersection of $\Gamma $ and $ T_aX +\mathbb{R}v$ is transversal, hence the results forms a submanifold of dimension $\dim_a X$. (B) Example \ref{Chazal}, the example by Chazal and Soufflet. Mark that, apart of the center of the sphere, points above the origin does not belong to the medial axis of $X$. }
    \label{fig:my_label}
\end{figure}
\begin{ex}(Chazal, Soufflet \cite{Chazal})\label{Chazal}
Consider $$X:=\partial \left( \mathbb{B}((0,0,2),2)\cup \lbrace y^2+(z-1)^2<1,\,x>0\rbrace\right) \subset \mathbb{R}^3.$$ 
Then for any point $x_t=(0,0,t)$ with $t\in[1,2)$ there is $m(x_t)=\lbrace 0\rbrace$. At the same time, $\tilde{r}_{v}(0)\leq \lim_{n\to\infty} r_v((1/n,0,0))=1$ where $v=\frac{x_t}{\|x_t\|}$. Therefore $x_t\in \overline{M_X}\backslash M_X$.

\end{ex}

\begin{rem}
Theorem \ref{Miurat} can be further generalised with virtually no change in the proof if we observe that only the existence of a neighbourhood $U$ of $x$ such that positive limit inferior of directional radii taken by the sequences in $m(U)$ is needed. With this approach, one can omit sequences of points that do not contribute actively to $M_X$ near $x$ (cf the points above the origin point of $X=\lbrace (y-x^2)(y-2x^2)=0\rbrace$ and $v=(0,1)$).
\end{rem}

\begin{rem} Theorem~\ref{Miurat} deserves also an exposition in the correspondence with our study of a tangent cone of the medial axes. Namely, for any $a\in X, v\in V_a$ we always have 
$$[a+\Tilde{r}_v(a),a+r_v(a)]\subset \overline{M}_X.$$
Even though the diameter of $m(x)$ is usually not separated from zero in the neighbourhood of the medial axis boundary, for all $x$ in an open segment $(a+r_v(a),a+\Tilde{r}_v(a))$ we are able to find a line $v\mathbb{R}$ in $C_xM_X$.
\end{rem}

As an example of an application of theorem~\ref{Miurat} we will prove a new result on the Birbrair-Denkowski reaching radius.

\begin{prop}\label{ciaglosc_promienia}
Weak reaching radius is continuous on $Reg_2 X$.
\end{prop}
\begin{proof}
We will prove $$\limsup_{X\backslash \lbrace a\rbrace \ni x\to a} r'(x)\leq r'(a)\leq \liminf_{X\backslash \lbrace a\rbrace \ni x\to a} r'(x).$$

A function $\rho(v)=r_v(a)$ is continuous on a compact set $V_a$, therefore there exists $v\in V_a$ with $$r_v(a)=r'(a).$$ 
Now for any sequence $X\ni x_\nu\to a$ with $r'(x_\nu)$ convergent to a certain $r$ and $v_\nu \in V_{x_\nu}$ convergent to $v$, we have
$$m(x_\nu+v_\nu r'(x_\nu))\ni x_\nu$$
and 
$$ x_\nu+v_\nu r'(x_\nu)\to a+vr.$$
From the upper semi-continuity of $m$ we derive 
$$\lbrace a\rbrace=\lim_{\nu\to\infty}\lbrace x_\nu\rbrace\subset\limsup_{\nu\to\infty} m(x_\nu+v_\nu r'(x_\nu))\subset m(a+vr).$$
Thus, $a\in m(a+vr)$, and consequently  $r\leq r_v(a)$, which proves $$\limsup_{X\backslash \lbrace a\rbrace \ni x\to a}  r'(x)\leq r'(a).$$

For the second inequality, write $d=\dim_aX$ and take $g$ a local parametrisation of $X$ at $a$. That is, an immersion $g:(G,0)\to (V\cap X,a)$ of $\mathcal{C}^2$ class with $G,V$ open subsets of $\mathbb{R}^d,\mathbb{R}^n$ respectively.

Since $V$ is open and $m(x)$ is upper semi-continuous, for any $r''<r'(a)$ and $v\in V_a$ there exists $U$ an open neighbourhood of $a+r''v$ such that $m(U)\subset V$. By summation, we can assume that $(a+rv)\in U$ for all $r\in[0,r'(a))$. 

Consider now 
$$ F:U\times G\ni (x,t)\to (<x-g(t),\frac{\partial g}{\partial t_i}(t)>)^d_{i=1}\in\mathbb{R}^d.$$
The function $F$ is $\mathcal{C}^1$ smooth and since $\lbrace \frac{\partial g}{\partial t_i}(t)\rbrace$ forms a base of $T_{g(t)}X$, there is $$F(x,t)=0\iff x-g(t)\in N_{g(t)}X.$$
In particular that brings $F(a+rv,0)=0$. Our goal now is to use the implicit function theorem to prove that $a+rv$ is separated from the medial axis.
Even though the determinant $\det \frac{\partial F}{\partial t}(a+rv,0)$ is not easily calculable, it is still a polynomial with respect to $r$ and  
$$\det \frac{\partial F}{\partial t}(a,0)=(-1)^d\sum( \det\frac{\partial(g_{i_1},\ldots,g_{i_d})}{\partial t}(0))^2\neq 0$$
hence it has only a finite number of zeros.
Because of that, we can find $r$ arbitrary close to $r'(a)$ with $\det \frac{\partial F}{\partial t}(a+rv,0)\neq 0$. Thus, from the implicit function theorem, there must exist $W\times T\subset U\times G$ a neighbourhood of $(a+rv,0)$, and a $\mathcal{C}^1$ smooth function $\tau :W\ni x\to\tau(x)\in T$ such that 
$$(F(x,t)=0\text{ and }(x,t)\in W\times T)\iff t=\tau(x).$$
That means $g(\tau(x))$ is a continuous selection from $m(x)$ on $W$, thus $m(x)$ is univalent on $W$ and as a consequence $a+rv$ is separated from $M_X$. From theorem~\ref{Miurat} we obtain $\Tilde{r}_v(a)>d(a+rv)=r$ thus
$$\forall v\in \Tilde{V}_a=V_a:\; \Tilde{r}_v(a)\geq r'(a).$$

Take now a sequence $X\ni x_\nu\to a$ with $\lim_{\nu\to\infty} r'(x_\nu)=\Dot{r}(a)$, and a sequence of normal directions $v_\nu\in V_{x_\nu}$ satisfying $r_{v_\nu}(x_\nu)=r'(x_\nu)$. Assuming without loss of generality that $\lbrace v_\nu\rbrace$ is convergent to some vector $v\in V_a$ we have then
$$\liminf_{X\backslash\lbrace a\rbrace \ni x\to a} r'(x)=\lim_{\nu\to\infty} r'(x_\nu)=\lim_{\nu\to\infty} r_{v_\nu}(x_\nu)\geq \tilde{r}_v(a)\geq r'(a).$$
\end{proof}

\begin{rem}

Proposition~\ref{ciaglosc_promienia} not only provides insights about Birbrair-Denkowski reaching radius. What is more, it allows to simplify the very definition of the Birbrair-Denkowski reaching radius by simply taking $$\dot{r}(a)=\liminf_{X\ni x\to a} r'(x).$$ 
\end{rem}

\begin{thm}\label{Denkowski=mine}
For any $a\in X$ the Birbrair-Denkowski reaching radius at $a$ is equal to $r(a)$.
\end{thm}
\begin{proof}

Take any sequence $\lbrace x_\nu\rbrace \subset X$ convergent to $a$ with weak reaching radii convergent to $\Dot{r}(a)$. For every $\nu\in\mathbb{N}$ take a sequence of $v^\nu_\mu\in V_{x_\nu}$ approximating  (sufficiently quickly) the weak reaching radius, say $|r_{v_\mu^\nu}(x_\nu)-r'(x_\nu)|<2^{-\mu}$. Picking from the $v_\nu^\nu$ a subsequence convergent to certain $v\in \Tilde{V}_a\subset\mathbb{S}$ we obtain 
$$\Tilde{r}_v(a)=\liminf\limits_{\substack{
     X\ni x\to a\\
     V_x\ni v_x\to v\in\Tilde{V}_a}} r_{v_x}(x)\leq\lim_{\nu\to\infty} r_{v_\nu^\nu}(x_\nu)=\Dot{r}(a).$$
Thus, by taking an infimum over $v\in\widetilde{V}_a$, we receive
$$r(a)\leq \Dot{r}(a).$$

Assume now that $r(a)< \Dot{r}(a)$. It is possible then to choose $\varepsilon>0$ such that 
$$r(a)+\varepsilon<\liminf_{x\to a} r'(x).$$
Thus we can find $U$ - such a neighbourhood of $a$ that for any $x\in U\cap X$
$$r(a)+\varepsilon/2<r'(x)\leq r_{v_x}(x),\;\forall v_x\in V_x.$$
Then by taking a sequence $\lbrace v_\nu\rbrace \subset\widetilde{V}_a$ realising an infimum in the definition of $r(a)$, we obtain for $x$ close enough to $a$

$$\Tilde{r}_{v_\nu}(a)=\liminf\limits_{\substack{
     X\ni x\to a\\
     V_x\ni v_x\to v_\nu\in\Tilde{V}_a}} r_{v_x}(x)> r(a)+\varepsilon/2.$$
By passing with $\nu$ to infinity, we obtain a contradiction

$$r(a)\geq r(a)+\varepsilon/2.$$
\end{proof}

\begin{rem}
Theorem~\ref{Denkowski=mine} and Proposition~\ref{ciaglosc_promienia} give together a continuity of the reaching radius on $Reg_2X$.
\end{rem}

\bibliographystyle{plain}
\bibliography{bibliography}

\begin{thebibliography}{10}

\bibitem{Cannarsa}
Luigi Ambrosio, Piermarco Cannarsa, and Halil~Mete Soner.
\newblock On the propagation of singularities of semi-convex functions.
\newblock {\em Annali della Scuola Normale Superiore di Pisa - Classe di
  Scienze}, Ser. 4, 20(4):597--616, 1993.

\bibitem{BigolinGrecko}
Francesco Bigolin and Gabriele~H. Greco.
\newblock Geometric characterizations of c1 manifolds in euclidean spaces by
  tangent cones.
\newblock {\em Journal of Mathematical Analysis and Applications}, 396(1):145
  -- 163, 2012.

\bibitem{BirbrairDenkowski}
Lev Birbrair and Maciej Denkowski.
\newblock Medial axis and singularieties.
\newblock {\em J.Geom. Anal.}, 27(3):2339--2380, 2017.

\bibitem{BirbrairSiersma}
Lev Birbrair and Dirk Siersma.
\newblock Metric properties of conflict set.
\newblock {\em Houston Mathematical Journal}, 35(1):73--80, 2009.

\bibitem{Blum}
Harry Blum.
\newblock {A} {T}ransformation for {E}xtracting {N}ew {D}escriptors of {S}hape.
\newblock In {\em Models for the Perception of Speech and Visual Form}, pages
  362--380. MIT Press, Cambridge, 1967.

\bibitem{Chazal}
Frédéric Chazal and R.~Soufflet.
\newblock Stability and finiteness properties of medial axis and skeleton.
\newblock {\em Journal of Dynamical and Control Systems}, 10:149--170, 04 2004.

\bibitem{Coste}
Michel Coste.
\newblock {\em An Introductionto O-minimal Geometry.}
\newblock Dip. Mat. Univ. Pisa, Dottorato di Ricerca in Matematica. Istituti
  Editoriali e Poligrafici Internazionali, Pisa, 2000.

\bibitem{DenkowskiOnPoints}
Maciej Denkowski.
\newblock On the points realizing the distance to the definable set.
\newblock {\em Journal of Mathematical Analysis and Applications},
  378:592--602, 2011.

\bibitem{DenkowskiLimits}
Maciej~P. Denkowski.
\newblock The kuratowski convergence of medial axes, 2016.

\bibitem{Erdos}
Paul Erdös.
\newblock Some remarks on the measurability of certain sets.
\newblock {\em Bull. Amer. Math. Soc.}, 51(10):728--731, 10 1945.

\bibitem{Erdos2}
Paul Erdös.
\newblock On the hausdorff dimension of some sets in euclidean space.
\newblock {\em Bull. Amer. Math. Soc.}, 52(2):107--109, 02 1946.

\bibitem{Fremlin}
DH~Fremlin.
\newblock Skeletons and central sets.
\newblock {\em Proceedings of the London Mathematical Society},
  74(3):701–720, 1997.

\bibitem{Miura}
Tatsuya Miura.
\newblock A characterization of cut locus for $c^1$ hypersurfaces.
\newblock {\em Nonlinear Differential Equations and Applications NoDEA}, 23(6),
  11 2016.

\bibitem{Rolin}
Jean-Philippe Rolin.
\newblock {A survey on o-minimal structures}.
\newblock {\em {Panoramas et synth{\`e}ses}}, 51:27--77, 2017.

\bibitem{RockafellarWets}
R.~Wets R.T.~Rockafellar.
\newblock {\em Variational analysis.}
\newblock Springer Verlag, 1998.

\bibitem{Mises}
Richard von Mises.
\newblock La base géométrique du théorème de m. mandelbrojt sur les points
  singuliers d'une fonction analytique.
\newblock {\em C, R. Acad. Sei. Paris}, 205:1353--1355, 1937.

\bibitem{Zajicek}
Ludek Zajicek.
\newblock Differentiability of the distance function and points of
  multi-valuedness of the metric projection in banach space.
\newblock {\em Czechoslovak Mathematical Journal}, 33(2):292--308, 1983.

\end{thebibliography}
\end{document}